\documentclass[a4paper,fleqn]{cas-sc}

\usepackage[numbers]{natbib}

\usepackage{amsmath,amssymb,amsthm}
\usepackage{mathtools}
\usepackage{enumitem}
\usepackage{aliascnt}
\usepackage{cleveref}

\newtheorem{theorem}{Theorem}[section]
\newaliascnt{lemma}{theorem}
\newtheorem{lemma}[lemma]{Lemma}
\aliascntresetthe{lemma}
\newaliascnt{proposition}{theorem}
\newtheorem{proposition}[proposition]{Proposition}
\aliascntresetthe{proposition}
\newaliascnt{corollary}{theorem}
\newtheorem{corollary}[corollary]{Corollary}
\aliascntresetthe{corollary}
\newaliascnt{claim}{theorem}

\aliascntresetthe{claim}
\theoremstyle{definition}
\newaliascnt{definition}{theorem}
\newtheorem{definition}[definition]{Definition}
\aliascntresetthe{definition}
\newaliascnt{example}{theorem}
\newtheorem{example}[example]{Example}
\aliascntresetthe{example}
\theoremstyle{remark}
\newaliascnt{remark}{theorem}
\newtheorem{remark}[remark]{Remark}
\aliascntresetthe{remark}

\crefname{theorem}{Theorem}{Theorems}
\Crefname{theorem}{Theorem}{Theorems}
\crefname{lemma}{Lemma}{Lemmas}
\Crefname{lemma}{Lemma}{Lemmas}
\crefname{proposition}{Proposition}{Propositions}
\Crefname{proposition}{Proposition}{Propositions}
\crefname{corollary}{Corollary}{Corollaries}
\Crefname{corollary}{Corollary}{Corollaries}
\crefname{claim}{Claim}{Claims}
\Crefname{claim}{Claim}{Claims}
\crefname{definition}{Definition}{Definitions}
\Crefname{definition}{Definition}{Definitions}
\crefname{example}{Example}{Examples}
\Crefname{example}{Example}{Examples}
\crefname{remark}{Remark}{Remarks}
\Crefname{remark}{Remark}{Remarks}

\newcommand{\B}{\{0,1\}}
\newcommand{\ZZ}{\mathbb{Z}}
\newcommand{\Zn}{\ZZ/n\ZZ}
\newcommand{\Zquot}[1]{\ZZ/#1\ZZ}
\newcommand{\rot}{\operatorname{rot}}
\newcommand{\gcdop}{\operatorname{\gcd}}
\newcommand{\One}{\mathbf{1}}
\newcommand{\Neck}[1]{\mathcal{N}(#1)}
\newcommand{\PrimNeck}[1]{\mathcal{P}(#1)}
\newcommand{\Ham}{\mathrm{d_H}}
\newcommand{\minDist}{\delta}
\newcommand{\Ktwo}{K_2}
\newcommand{\Dtwo}{\mathcal{D}_2}
\newcommand{\DtwoHT}{\mathcal{D}_2^{\mathrm{HT}}}
\newcommand{\DtwoSP}{\mathcal{D}_2^{\mathrm{SP}}}
\newcommand{\DtwoSUB}{\mathcal{D}_2^{\mathrm{SUB}}}
\newcommand{\DtwoMP}{\mathcal{D}_2^{\mathrm{MP}}}
\newcommand{\tautwo}{\tau_2}

\newcommand{\PrimCt}[1]{\mathsf{P}(#1)}

\begin{document}
\let\WriteBookmarks\relax
\let\printorcid\relax
\ExplSyntaxOn
\cs_if_exist:NT \g_stm_nologo_bool
  { \bool_gset_true:N \g_stm_nologo_bool }
\ExplSyntaxOff

\ExplSyntaxOn
\cs_set:Npn \__first_footerline:
  {
    \group_begin:
    \small
    \sffamily
    \ifnum\theblind>0\relax
    \else
      \__short_authors:
    \fi
    \group_end:
  }
\ExplSyntaxOff

\renewcommand{\floatpagefraction}{1}
\renewcommand{\textfraction}{.001}

\shorttitle{Binary necklaces with minimum nontrivial rotational Hamming distance $2$}
\shortauthors{Makhynko}

\title[mode=title]{Binary necklaces with minimum nontrivial rotational Hamming distance $2$}

\author[1]{Mykola Makhynko}
\cormark[1]
\ead{mykola.makhynko@gmail.com}
\affiliation[1]{organization={Independent Researcher},
            city={Unionville},
            state={Ontario},
            country={Canada}}
\cortext[1]{Corresponding author}

\begin{abstract}
We classify and count binary necklaces $[w]$ for which the representative-independent minimum nontrivial rotational Hamming distance $\minDist(w)$ equals $2$, the first positive $\minDist$ layer after the periodic case.  With a fixed priority convention, the full set of nonzero shifts attaining the value $2$ determines a classification into half-turn (HT), single-pair (SP), and subgroup-pattern (SUB) classes, with the residual multi-pair (MP) class empty.  We give closed formulas for the HT, SP, and SUB counts and an explicit M\"obius--totient divisor-sum formula for $|\mathcal D_2(n)|$.  The larger subgroup-shaped branch is distinct from the SUB priority class and includes the SUB cases together with the boundary cases assigned by priority to HT and SP.  A quotient--background normal form describes this larger branch and records the corresponding boundary counts.  For odd prime lengths $p$, every distance-$2$ necklace is rotation-equivalent to an arithmetic interval on the prime cycle, and $|\mathcal D_2(p)|=(p^2-4p+7)/2$.
\end{abstract}

\maketitle

\section{Introduction}
\label{sec:intro}

A binary necklace of length $n$ is a binary word of length $n$ considered up to cyclic rotation.  For a word $w\in\{0,1\}^n$, let
\[
 d_w(k)=\Ham(w,\rot^k w),\qquad 0\le k<n,
\]
where $(\rot^k w)_i=w_{i-k}$.  For $n\ge2$ define
\[
 \minDist(w)=\min_{1\le k<n} d_w(k).
\]
All rotational Hamming distances of a binary word are even, and $\minDist(w)>0$ is equivalent to primitivity.  In the language of conjugates, this parity boundary was studied by Shallit~\cite{Shallit2009}; the first positive exact-minimum case here is therefore $\minDist(w)=2$.
Throughout, $\Zn$ denotes $\ZZ/n\ZZ$, and indices and shifts are read modulo $n$.

Under the encoding $x_i=(-1)^{w_i}$, define the periodic autocorrelation
\[
 C_x(k):=\sum_{i\in\Zn}x_i x_{i-k}.
\]
Then
\[
 d_w(k)=\frac{n-C_x(k)}{2}.
\]
Thus rotational Hamming distance can also be viewed as an off-peak periodic-autocorrelation parameter familiar from sequence design.  For periodic-autocorrelation terminology and sequence-design background, see, for example, \cite{GolombGong2005,Sarwate1979}; no theorem from these references is used below.  Consequently, $\minDist(w)=2$ means that some nonzero shift $k$ attains $C_x(k)=n-4$, and no nonzero shift has a larger autocorrelation value.  In this sense the layer is complementary to the usual sequence-design problem of constructing sequences with small or controlled off-peak autocorrelation.  We use this interpretation only as background: the proofs below use cyclic word and necklace structure directly.

\subsection{Motivation and contribution}
\label{subsec:motivation-contribution}

A natural refinement of classical necklace enumeration is to enumerate necklaces by their minimum nontrivial rotational Hamming distance.  P\'olya's formula counts all binary necklaces, while the classical primitive-necklace formula separates the periodic boundary: for $n\ge2$, $\minDist(w)=0$ exactly when a nontrivial rotation fixes $w$, and $\minDist(w)>0$ exactly when the necklace is primitive.  Since binary rotational Hamming distances are even, the first positive exact layer in this stratification is $\minDist(w)=2$.

The distance-$2$ layer is therefore a test case for exact-distance enumeration beyond the periodic/primitive boundary.  It is also the first place where prescribed-shift information is not enough.  Counting words for which one prescribed shift has distance $2$ is a projection of the problem; it does not determine how all distance-$2$ shifts occur together on the same necklace, nor does it give a disjoint necklace-level enumeration.  The aim of this paper is to recover that full minimizer-set structure and then count the resulting disjoint classes.

These comparisons also clarify the novelty boundary.  Gabri\'c's work~\cite{Gabric2022} gives related aggregate word-level and Lyndon-level formulas.  Sections S5.2, ``Prescribed-shift counts'', and S5.3, ``Aggregate and pure half-turn comparison with Gabri\'c'', of the supplementary material review the prescribed-shift and aggregate comparisons.  The present paper contributes the necklace-level classification by the complete minimizer set, the proof that the residual multi-pair case is empty at distance $2$, and the classwise HT/SP/SUB formulas whose sum agrees with those aggregate comparisons.

Let $\Neck{n}$ denote the set of binary necklaces of length $n$.
This paper classifies and counts the necklaces
\[
\Dtwo(n)=\{[w]\in\Neck{n}:\minDist(w)=2\}.
\]
The proof is based on the complete set of minimizing shifts
\[
\Ktwo(w)=\{k\in\{1,\ldots,n-1\}:d_w(k)=2\},
\]
rather than on a single shift witnessing distance $2$.  A prescribed-shift count counts words with $d_w(k)=2$ for one prescribed shift $k$.  The exact-minimum problem counts necklaces while determining the full minimizer set and requiring that every nonzero shift have distance at least $2$; in particular, no nonzero shift may have distance $0$.  The main theorem solves this exact-minimum necklace problem.  This distinction is essential: one prescribed shift can have distance $2$ while the full set $\Ktwo(w)$ has different orbitwise shapes.

The class labels used in the theorem are introduced for this enumerative purpose.  They are not independent structural axioms that are assumed to be disjoint in advance.  The raw properties can overlap: a subgroup-shaped minimizer set may contain the half-turn, and a subgroup of quotient order $3$ has exactly one antipodal pair of nonzero shifts.  We therefore assign classes by a fixed priority convention.  The convention is part of the definition and is what turns the structural alternatives into disjoint summands.

The main structural result is that, under this priority convention, the four-way HT/SP/SUB/MP taxonomy collapses in the distance-$2$ case.  Every distance-$2$ necklace belongs to exactly one of three classes:
\[
\Dtwo(n)=\DtwoHT(n)\sqcup\DtwoSP(n)\sqcup\DtwoSUB(n),
\qquad
\DtwoMP(n)=\varnothing.
\]
Here HT means that the half-turn is a minimizer, SP means that the minimizer set is one antipodal pair, and SUB means that the minimizer set is the nonzero part of a cyclic subgroup after the HT and SP priority cases have been removed.  We use subgroup-pattern only for the SUB priority class.  We use subgroup-shaped for the larger normal-form condition that the full minimizer set is the nonzero part of a cyclic subgroup; this condition covers the mixed half-turn and order-$3$ SP boundary cases as well as SUB cases.

For the compact class formulas, let
\[
\PrimCt{\ell}:=\sum_{d\mid \ell}\mu(d)2^{\ell/d},
\]
be the number of primitive binary words of length $\ell$.  In all terms below that are multiplied by indicators, the expression following an indicator is evaluated only when the indicator is $1$; otherwise the entire term is interpreted as zero.  For instance, $\One_{3\mid n}\PrimCt{n/3}$ is zero when $3\nmid n$.  At $n=1$, the empty-class convention stated formally in \Cref{sec:prelim} applies: $\Dtwo(1)=\DtwoHT(1)=\DtwoSP(1)=\DtwoSUB(1)=\DtwoMP(1)=\varnothing$.

For the theorem statement, these class symbols use the following priority convention.  For a necklace $\mathcal N=[w]\in\Dtwo(n)$, let $K=\Ktwo(w)$; this set is independent of the chosen representative.  Regard $K$ as a subset of $\Zn\setminus\{0\}$.  If $K=\{k_1,\ldots,k_r\}$, write $\gcdop(n,K):=\gcdop(n,k_1,\ldots,k_r)$.
Assign the class in priority order.  First, put $\mathcal N\in\DtwoHT(n)$ when $n$ is even and $n/2\in K$.  If this HT case does not occur, put $\mathcal N\in\DtwoSP(n)$ when $K=\{k,n-k\}$ for some $1\le k<n/2$.  If neither HT nor SP occurs, put $\mathcal N\in\DtwoSUB(n)$ when, for $\tau=\gcdop(n,K)$,
\[
K=\tau\Zn\setminus\{0\},
\]
where $\tau\Zn\setminus\{0\}$ denotes the nonzero residue classes in $\Zn$ represented by integer multiples of $\tau$.  The residual class $\DtwoMP(n)$ consists of all remaining cases.  The main enumeration theorem gives the classwise formulas and the aggregate count.  In the aggregate display we expand the primitive-word summands so that the final formula uses only $\mu$, $\varphi$, powers of $2$, indicators, and divisor sums; the compact $\PrimCt{\cdot}$-notation version follows it.

\begin{theorem}[Main enumeration theorem]
\label{thm:main}
For every $n\ge1$,
\[
\Dtwo(n)=\DtwoHT(n)\dot\cup\DtwoSP(n)\dot\cup\DtwoSUB(n),
\qquad
\DtwoMP(n)=\varnothing.
\]
The three class sizes are
\[
\begin{aligned}
|\DtwoHT(n)|
&=\One_{2\mid n}\,2^{\frac n2-1},\\
|\DtwoSP(n)|
&=\One_{3\mid n}\PrimCt{n/3}
  +\sum_{\substack{m\mid n\\ m\ge4}}
    \frac{\varphi(m)}2\,2^{\frac nm-1}(m-3),\\
|\DtwoSUB(n)|
&=\sum_{\substack{\tau\mid n\\ (n/\tau)\ \mathrm{odd}\\ n/\tau\ge5}}
  \PrimCt{\tau}
\end{aligned}.
\]
Consequently, the total count has the following explicit M\"obius--totient divisor-sum form:
\[
|\Dtwo(n)|
=\One_{2\mid n}\,2^{\frac n2-1}
+\One_{3\mid n}\sum_{d\mid (n/3)}\mu(d)\,2^{\frac{n}{3d}}
+\sum_{\substack{m\mid n\\ m\ge4}}
  \frac{\varphi(m)}2\,2^{\frac nm-1}(m-3)
+\sum_{\substack{\tau\mid n\\ (n/\tau)\ \mathrm{odd}\\ n/\tau\ge5}}
  \sum_{d\mid \tau}\mu(d)\,2^{\tau/d}.
\]
Equivalently, in the compact primitive-word notation used in the class formulas,
\[
|\Dtwo(n)|
=\One_{2\mid n}\,2^{\frac n2-1}
+\One_{3\mid n}\PrimCt{n/3}
+\sum_{\substack{m\mid n\\ m\ge4}}
  \frac{\varphi(m)}2\,2^{\frac nm-1}(m-3)
+\sum_{\substack{\tau\mid n\\ (n/\tau)\ \mathrm{odd}\\ n/\tau\ge5}}
  \PrimCt{\tau}.
\]
\end{theorem}

The proof has four steps.
\begin{enumerate}[label=(\roman*),itemsep=0pt,topsep=2pt]
\item \Cref{sec:prelim} fixes notation and defines the distance-$2$ classes.
\item \Cref{sec:structure} proves that a prescribed distance-$2$ shift has exactly one active orbit, and that this orbit is a cyclic interval.
\item \Cref{sec:trichotomy} uses compatibility among minimizing shifts to force the minimizer set into the HT/SP/SUB alternatives.
\item \Cref{sec:defect-background} gives the single-defect normal form from primitive quotient--background words, while \Cref{sec:enum} counts the three classes and derives the prime-length corollary.
\end{enumerate}
The formal proof of \Cref{thm:main} appears at the end of \Cref{sec:enum}.  At that point, \Cref{thm:countHT,thm:countSP,thm:countSUB,cor:countD2} have established the HT, SP, and SUB summands and the aggregate count.

The prime-length corollary has a simple form: if $p$ is an odd prime, then every distance-$2$ necklace is rotation-equivalent to an arithmetic interval on the prime cycle and
\[
|\Dtwo(p)|=\frac{p^2-4p+7}{2}.
\]

The word-level boundary for conjugates was studied by Shallit \cite{Shallit2009}: over a binary alphabet, the Hamming distance between conjugates is even, so distance $2$ is the first positive case.  Gabri\'c's Lemma~7 gives the prescribed-shift word-level comparison.  Theorems~17 and~19 give the aggregate word-level and Lyndon-level formulas, and Theorem~18 gives the pure half-turn boundary comparison \cite[Lem.~7, Thms.~17, 18, and~19]{Gabric2022}.  After primitive-orbit division, the aggregate comparisons agree with the aggregate total above.  The added necklace-level contribution here is the recovery of the full minimizer set $\Ktwo(w)$ and the induced HT/SP/SUB split.
Sections S5.2 and S5.3 of the supplementary material give these external comparisons.

The supplementary material contains optional reformulations and algebraic comparisons with aggregate word-level formulas in the literature.  Section S4, ``Practical classifier and worked examples'', gives examples and classifier details.  Section S5.1, ``Finite check for small $n$'', inside Section S5, ``Finite check and algebraic comparisons'', gives the deterministic enumeration scan for $2\le n\le 20$ together with pseudocode, the aggregate count table, sample records, a field specification, and the aggregate output summary.  These details document the finite check.

The finite check documented in the supplementary material produces the small-length table for $2\le n\le 20$ by an orbit scan that uses the full distance profile over nonzero shifts before assigning any class label and returns zero residual MP counts over its checked range.  Section S5.1 records the aggregate output and the per-necklace fields used by the scan for the $2063$ retained records.

\section{Preliminaries}
\label{sec:prelim}

\subsection{Words, rotations, and necklaces}
A \emph{binary word} of length $n$ is a vector $w=(w_0,\dots,w_{n-1})\in\B^n$.
We view indices modulo $n$.

For $k\in\Zn$, define the cyclic rotation
\[ (\rot^k w)_i := w_{i-k}\qquad (i\in\Zn). \]
Two words are \emph{rotation-equivalent} if one is obtained from the other by a cyclic rotation.
A \emph{binary necklace} of length $n$ is a rotation-equivalence class.
A word is \emph{primitive} or \emph{aperiodic} if no nonzero rotation fixes it, equivalently if its rotation orbit has size $n$.
Throughout, we count necklaces up to rotation only; we do \emph{not} identify a word with its reverse or with its binary complement.

\subsection{Rotational Hamming distance}
The Hamming distance between two length-$n$ words is
\[ \Ham(u,v) := |\{i\in\Zn : u_i\neq v_i\}|. \]
For a word $w$ and shift $k\in\{0,\dots,n-1\}$ define the \emph{rotational Hamming distance}
\begin{equation}
\label{eq:dwdef}
d_w(k) := \Ham\bigl(w,\rot^k w\bigr).
\end{equation}
We always have $d_w(0)=0$ and $d_w(k)=d_w(n-k)$.
Together with the parity argument in \Cref{lem:dw-even}, this symmetry implies that $\minDist(w)$ can only take even values in the binary case.
For $n\ge2$, define the minimum nontrivial rotational Hamming distance
\begin{equation}
\label{eq:mindist}
\minDist(w) := \min_{1\le k\le n-1} d_w(k).
\end{equation}
For $n=1$ there is no nontrivial shift, so $\minDist$ is not used; the distance-$2$ count convention is recorded below in the definition of $\Dtwo(1)$.
When $n\ge2$ and $\minDist(w)=0$, the word is periodic.
In particular, for $n\ge2$ the constant words $0^n$ and $1^n$ satisfy $\minDist(w)=0$ since every nontrivial rotation fixes them.

\begin{lemma}[Parity of rotational Hamming distances]
\label{lem:dw-even}
For every binary word $w\in\B^n$ and every $k\in\{0,\dots,n-1\}$, the value $d_w(k)$ is even.
In particular, when $n\ge2$, $\minDist(w)$ is always an even integer.
\end{lemma}

\begin{proof}
Fix $k$ and let $g=\gcdop(n,k)$.
The permutation $i\mapsto i+k$ decomposes $\Zn$ into $g$ disjoint cycles, each of length $n/g$.
Along one cycle $(i_0,i_0+k,i_0+2k,\dots)$, the contribution to $d_w(k)$ equals the number of indices $t$ for which
$w_{i_t}\neq w_{i_{t+1}}$ (indices taken modulo the cycle length).
A cyclic binary sequence changes value an even number of times around the cycle, so each cycle contributes an even number.
Summing over all cycles, $d_w(k)$ is even.
\end{proof}

\begin{lemma}[Positive minimum and primitive-orbit division]
\label{lem:positive-minimum-primitive}
Let $n\ge2$ and let $w\in\B^n$.
Then $\minDist(w)>0$ if and only if no nonzero rotation fixes $w$; equivalently, $w$ is primitive and its rotation orbit has exactly $n$ words.
Consequently, if $X\subseteq\B^n$ is rotation-invariant and every $w\in X$ satisfies $\minDist(w)>0$, then the number of necklaces represented by $X$ is $|X|/n$.
In particular, primitive-orbit division by $n$ applies to the distance-$2$ word family considered in this paper, $\{w\in\B^n:\minDist(w)=2\}$.
Moreover, any word with $d_w(k)=2$ for some nonzero shift $k$ is primitive.
\end{lemma}

\begin{proof}
For $1\le s<n$, the equality $d_w(s)=0$ is exactly the equality $\rot^s w=w$.
Thus $\minDist(w)>0$ is equivalent to a trivial stabilizer under the cyclic rotation action, and the orbit-stabilizer theorem gives orbit size $n$ exactly in this case.
A nontrivial stabilizing rotation is the same as a proper period, so trivial stabilizer is equivalent to primitivity.
If $X$ is rotation-invariant and every $w\in X$ has $\minDist(w)>0$, then $X$ is a disjoint union of rotation orbits of size $n$, so its necklace quotient has cardinality $|X|/n$.
The distance-$2$ assertion is the special case $X=\{w\in\B^n:\minDist(w)=2\}$.
For the final assertion, suppose that $w$ has a proper rotational period $p<n$, so $w$ is obtained by repeating a length-$p$ block $t=n/p>1$ times.
For any shift $k$, the distance $d_w(k)$ is $t$ times the corresponding rotational Hamming distance of the period-$p$ block at the induced shift modulo $p$.
By \Cref{lem:dw-even} applied to that block, this induced distance is either $0$ or at least $2$; hence $d_w(k)$ is either $0$ or at least $2t\ge4$ and cannot equal $2$.
\end{proof}

\begin{lemma}[Cycle decomposition for a prescribed shift]
\label{lem:shift-cycle-decomp}
Fix $w\in\B^n$ and $k\in\{0,\dots,n-1\}$, and write $g:=\gcdop(n,k)$ and $m:=n/g$.
For each residue $r\in\{0,\dots,g-1\}$ define a length-$m$ word $u^{(r)}\in\B^m$ by
\[
u^{(r)}_t := w_{r+tk}\qquad (t\in\Zquot{m}),
\]
where indices on the right are taken in $\Zn$.
Then
\[
d_w(k)=\sum_{r=0}^{g-1} \Ham\bigl(u^{(r)},\rot u^{(r)}\bigr)
      =\sum_{r=0}^{g-1}\bigl|\{t\in\Zquot{m}: u^{(r)}_t\neq u^{(r)}_{t-1}\}\bigr|.
\]
In other words, $d_w(k)$ is the total number of value-changes around the $g$ disjoint $k$-orbits of $\Zn$.
\end{lemma}

\begin{proof}
The permutation $i\mapsto i+k$ partitions $\Zn$ into $g$ cycles, each of length $m$.
The cycle corresponding to $r$ is $(r,r+k,r+2k,\dots,r+(m-1)k)$.
For an index $i=r+tk$ we have $(\rot^k w)_i=w_{i-k}=w_{r+(t-1)k}$, so the indicator of the event
$w_i\neq (\rot^k w)_i$ is exactly the indicator of $u^{(r)}_t\neq u^{(r)}_{t-1}$.
Summing over $t\in\Zquot{m}$ gives the second expression for the contribution of the $r$th cycle, and summing over all cycles yields the claim.
\end{proof}

The present paper studies the next case, $\minDist(w)=2$.

For every word $w\in\B^n$ with $n\ge2$, define the distance-$2$ shift set
\begin{equation}
\label{eq:K2def}
\Ktwo(w) := \{k\in\{1,\dots,n-1\} : d_w(k)=2\}.
\end{equation}
When $\minDist(w)=2$, this is exactly the set of nontrivial minimizing shifts.
The set $\Ktwo(w)$ depends only on the necklace, because $d_w(k)$ is rotation-invariant.

\begin{remark}[Rotation-invariant word notation on necklace classes]
\label{rem:necklace-word-invariants}
If $\mathcal N=[w]$ is a binary necklace of length $n\ge2$, then the rotation-invariance of $d_w(k)$ lets us write
\[
\minDist(\mathcal N):=\minDist(w),
\qquad
\Ktwo(\mathcal N):=\Ktwo(w),
\]
for any representative $w\in\mathcal N$.
Accordingly, every quantity used below that is built only from the rotation-invariant data $d_w(k)$ or $\Ktwo(w)$ is used on necklace classes by the same representative-independent convention.
\end{remark}

\begin{remark}[Set notation for minimizing shifts]
When convenient, we identify a shift $k\in\{1,\dots,n-1\}$ with its residue class in $\Zn$.
Accordingly, for a nonempty set $K\subseteq\{1,\dots,n-1\}$ we freely regard $K$ as a subset of $\Zn\setminus\{0\}$ and write
\[
\gcdop(n,K):=\gcdop(n,k_1,\dots,k_r),
\]
when $K=\{k_1,\dots,k_r\}$.
With this convention, we use the displayed notation
\[
K=\tau\Zn\setminus\{0\}.
\]
The notation means that $K$ is exactly the set of nonzero residue classes in $\Zn$ represented by integer multiples of $\tau$.
\end{remark}

When $\Ktwo(w)$ is nonempty for a word of length $n$, write
\[
\tautwo(w):=\gcdop(n,\Ktwo(w)).
\]
By the rotation-invariant convention of \Cref{rem:necklace-word-invariants}, we also write
\[
\tautwo(\mathcal N):=\gcdop(n,\Ktwo(\mathcal N)),
\]
for the corresponding necklace class.  In the refined HT and SUB slices introduced below, this is the subgroup divisor used in the notation $\DtwoHT(n;\tau)$ and $\DtwoSUB(n;\tau)$.

\subsection{Distance-\texorpdfstring{$2$}{2} class and slice notation}
\label{subsec:d2-branches}
The following distance-$2$ class names record the shape of the full set $\Ktwo(w)$ of minimizing shifts, not merely the presence of a prescribed distance-$2$ witness.
We reserve \emph{class} for the priority labels HT, SP, SUB, and MP.  Refinements such as $\DtwoHT(n;\tau)$, $\DtwoSP(n;m)$, and $\DtwoSUB(n;\tau)$ are called \emph{slices}, while normal-form branches refer only to proof alternatives.

The classes are assigned by priority: half-turn first, then single-pair, then subgroup-pattern, with a residual multi-pair class for the remaining cases.  This priority convention is both an enumeration device and a definition: without it, the natural structural properties would overlap.  A subgroup-shaped minimizer set can contain the half-turn, and the order-$3$ SP boundary case is a subgroup-shaped set that consists of one antipodal pair; these boundary cases are assigned to HT or SP before SUB is considered.

We use subgroup-pattern only for the priority class $\DtwoSUB$.  We use subgroup-shaped for the larger normal-form condition that the full minimizer set has the form $\tau\Zn\setminus\{0\}$ for some divisor $\tau\mid n$.  This condition covers mixed half-turn cases, order-$3$ SP boundary cases, and SUB cases.  \Cref{thm:trichotomy} proves that the residual class is empty at distance $2$.

\begin{definition}[Distance-$2$ class taxonomy]
\label{def:d2-taxonomy}
Let $n\ge2$ and let $\mathcal N$ be a necklace of length $n$ with $\minDist(\mathcal N)=2$.  Set $K:=\Ktwo(\mathcal N)$.
The class labels are assigned as follows.
\begin{enumerate}[label=\textup{(\roman*)},leftmargin=2.5em]
\item \textbf{HT:} $n$ is even and $n/2\in K$.
\item \textbf{SP:} the HT case does not occur and $K=\{k,n-k\}$ for some $1\le k<n/2$.
\item \textbf{SUB:} neither HT nor SP occurs and, with $\tau:=\gcdop(n,K)$,
\[
K=\tau\Zn\setminus\{0\}.
\]
\item \textbf{MP:} all remaining cases; equivalently, after the half-turn and single-pair cases have been removed, $K$ contains at least two antipodal pairs but is not of the subgroup-shaped form $\tau\Zn\setminus\{0\}$.
\end{enumerate}
\end{definition}

For concrete orientation, the following small representatives illustrate the class labels; these rows are orientation examples only, not additional hypotheses or proof inputs.  The set $\Ktwo(w)$ is listed as a subset of $\Zn\setminus\{0\}$ for the displayed length.
\begin{center}
\small
\begin{tabular}{ccccc}
\hline
Class & Length $n$ & Representative $w$ & $\Ktwo(w)$ & Classifying feature \\
\hline
HT & $8$ & \texttt{00010011} & $\{4\}$ & half-turn $4=n/2$ \\
SP & $5$ & \texttt{11000} & $\{1,4\}$ & one antipodal pair \\
SP & $3$ & \texttt{100} & $\{1,2\}=\Zquot{3}\setminus\{0\}$ & order-$3$ SP boundary case \\
SUB & $5$ & \texttt{10000} & $\{1,2,3,4\}$ & $\tau=1$, $\Ktwo(w)=\Zn\setminus\{0\}$ \\
\hline
\end{tabular}
\end{center}
The priority convention also has mixed boundary cases.  The length-$3$ representative \texttt{100} has
$\Ktwo(w)=\{1,2\}=\Zquot{3}\setminus\{0\}$: it is subgroup-shaped with quotient order $M=3$, but the HT condition is not satisfied and $|\Ktwo(w)|=2$, so the necklace is assigned to SP before SUB is considered.  In the supplementary worked example at length $12$, the representative \texttt{101001001001} has
$\Ktwo(w)=\{3,6,9\}$; because $6=n/2$ lies in this subgroup-shaped minimizer set, the necklace is assigned to HT before SUB is considered.  Section S4 of the supplementary material, ``Practical classifier and worked examples'', gives the full classifier trace, additional context for these representatives, finite-check rows for the boundary cases, and the worked quotient--background construction using these class labels.

\subsection{Classical necklace counts (for comparison)}
The total number of binary necklaces of length $n$ is given by P\'olya enumeration:
\begin{equation}
\label{eq:necklaces}
|\Neck{n}| = \frac{1}{n}\sum_{d\mid n} \varphi(d)\,2^{n/d}.
\end{equation}
The number of \emph{primitive} (aperiodic) binary necklaces of length $n$ is
\begin{equation}
\label{eq:prim-necklaces}
\PrimNeck{n} = \frac{1}{n}\sum_{d\mid n} \mu(d)\,2^{n/d}.
\end{equation}
We will also use the classical count of primitive binary \emph{words} (not modulo rotation):
\begin{equation}
\label{eq:prim-words}
\#\{\text{primitive words of length }\ell\} = \sum_{d\mid \ell} \mu(d)\,2^{\ell/d}.
\end{equation}
Here a binary word of length $\ell$ has \emph{minimal period} $p\mid \ell$ if $p$ is the least positive divisor of $\ell$ for which the word is obtained by repeating a word of length $p$ exactly $\ell/p$ times; equivalently, a word of length $\ell$ is primitive exactly when its minimal period is $\ell$, matching the rotation-orbit convention above.

\begin{remark}
Equations \eqref{eq:necklaces}--\eqref{eq:prim-necklaces} are classical: P\'olya supplies the cycle-index enumeration viewpoint, and Lothaire supplies standard word and necklace terminology, including primitive-word background \cite{Polya,Lothaire}.
We include these formulas to highlight similarities with the enumeration formulas below: all counts ultimately reduce to divisor sums with $\varphi$ or $\mu$.
\end{remark}

\begin{remark}[Word, necklace, and Lyndon comparisons]
\label{rem:comparison-bridge}
The Gabri\'c comparison formulas used here are located in Sections S5.2 and S5.3 of the supplementary material. They use word-level counts together with the usual passage from primitive words to rotation orbits: necklaces are the orbits, while Lyndon words are standard representatives of primitive orbits. For the distance-$2$ word families considered here, this passage is formalized in \Cref{lem:positive-minimum-primitive}. These formulas provide numerical comparisons and are not part of the HT/SP/SUB proofs.
\end{remark}

\subsection{The distance-\texorpdfstring{$2$}{2} minimizer-set classes}
We now define the distance-$2$ classes used in the enumeration.
They are the HT, SP, and SUB parts of the distance-$2$ class taxonomy in \Cref{def:d2-taxonomy}; the residual multi-pair part is retained as a named residual class and is proved empty in \Cref{thm:trichotomy}.
All named class definitions in this subsection inherit the priority convention of \Cref{def:d2-taxonomy}.  HT membership is tested first.  SP is tested only after HT has failed, and SUB is tested only after both HT and SP have failed.  The term subgroup-pattern is reserved for the priority class $\DtwoSUB$, whereas subgroup-shaped denotes the larger normal-form condition before this priority assignment.

\begin{definition}[Distance-$2$ necklaces]
For $n\ge 2$, let
\[
\Dtwo(n) := \{\mathcal N\in\Neck{n}: \minDist(\mathcal N)=2\}.
\]
For $n=1$, set
\[
\Dtwo(1)=\DtwoHT(1)=\DtwoSP(1)=\DtwoSUB(1)=\DtwoMP(1):=\varnothing.
\]
\end{definition}

\begin{definition}[Half-turn class ($\DtwoHT$)]
\label{def:DtwoHT}
Assume $n$ is even.
For a necklace class $\mathcal N\in\Dtwo(n)$, write $\mathcal N\in\DtwoHT(n)$ if the half-turn is a minimizer:
\[ n/2\in\Ktwo(\mathcal N). \]
If $n$ is odd, set $\DtwoHT(n)=\emptyset$.
\end{definition}

\begin{definition}[Refined HT slices $\DtwoHT(n;\tau)$]
Assume $n$ is even and let $\tau\mid (n/2)$.
For a necklace class $\mathcal N\in\DtwoHT(n)$, write $\mathcal N\in\DtwoHT(n;\tau)$ if
\[
\Ktwo(\mathcal N)=\tau\Zn\setminus\{0\}.
\]
Note that $\tau=n/2$ corresponds to the ``pure'' half-turn case $\Ktwo(\mathcal N)=\{n/2\}$.
\end{definition}

\begin{definition}[Single-pair class ($\DtwoSP$)]
\label{def:DtwoSP}
For a necklace class $\mathcal N\in\Dtwo(n)$, write $\mathcal N\in\DtwoSP(n)$ if
\[ \Ktwo(\mathcal N)=\{k,n-k\} \quad\text{for some } 1\le k < n/2. \]
Equivalently, $|\Ktwo(\mathcal N)|=2$ and (when $n$ is even) $n/2\notin\Ktwo(\mathcal N)$.
\end{definition}

\begin{definition}[Subgroup-pattern class ($\DtwoSUB$)]
\label{def:DtwoSUB}
For a necklace class $\mathcal N\in\Dtwo(n)$, write $\mathcal N\in\DtwoSUB(n)$ if
\[ \Ktwo(\mathcal N)=\tau\Zn\setminus\{0\} \quad\text{for some divisor }\tau\mid n, \]
with the additional constraints
\begin{enumerate}[label=(\alph*),leftmargin=2.5em]
\item $|\Ktwo(\mathcal N)|>2$ (more than one antipodal pair), and
\item if $n$ is even, then $n/2\notin\Ktwo(\mathcal N)$ (we treat the half-turn separately in class $\DtwoHT$; see \Cref{def:DtwoHT}).
\end{enumerate}
The divisor $\tau$ is called the \emph{strict period} of $\mathcal N$.
\end{definition}

\begin{definition}[Residual multi-pair class ($\DtwoMP$)]
\label{def:DtwoMP}
For a necklace class $\mathcal N\in\Dtwo(n)$, write $\mathcal N\in\DtwoMP(n)$ if $\mathcal N$ belongs to none of $\DtwoHT(n)$, $\DtwoSP(n)$, and $\DtwoSUB(n)$.
The distance-$2$ collapse theorem \Cref{thm:trichotomy} proves that $\DtwoMP(n)$ is empty for every $n$.
\end{definition}

\begin{remark}[Strict period vs.~rotational period]
In this manuscript, \emph{strict period} is reserved for the SUB refinement in \Cref{def:DtwoSUB}.  Thus, when $\mathcal N\in\DtwoSUB(n)$ and $\Ktwo(\mathcal N)=\tau\Zn\setminus\{0\}$, the divisor $\tau$ is the strict period of $\mathcal N$.  In subgroup-shaped HT and order-$3$ SP boundary cases outside SUB, the same divisor is only the subgroup divisor.  It should not be confused with a rotational period of a representative word $w$ (which would mean $\rot^k w=w$ for some $k\neq 0$ and hence $\minDist(w)=0$).
\end{remark}

\begin{definition}[SP by order]
\label{def:DtwoSPByOrder}
For $k\in\Zn$, write $\operatorname{ord}_{\Zn}(k)$ for the least positive integer $m$ such that $mk=0$ in $\Zn$; equivalently, $m=|\langle k\rangle|$.
For each integer $m\ge 3$, define $\DtwoSP(n;m)$ as follows.
If $m$ divides $n$, let $\DtwoSP(n;m)$ be the set of necklace classes $\mathcal N\in\DtwoSP(n)$ for which there exists a shift $k$ with
$\Ktwo(\mathcal N)=\{k,n-k\}$ and $\operatorname{ord}_{\Zn}(k)=m$.
If $m$ does not divide $n$, set $\DtwoSP(n;m)=\varnothing$.
\end{definition}

\begin{definition}[SUB by strict period]
For $\tau\mid n$, let $\DtwoSUB(n;\tau)$ be the set of necklace classes $\mathcal N\in\DtwoSUB(n)$ for which
$\Ktwo(\mathcal N)=\tau\Zn\setminus\{0\}$ (equivalently, $\tau$ is the strict period of $\mathcal N$).
\end{definition}

\section{Prescribed-shift orbit geometry at distance \texorpdfstring{$2$}{2}}
\label{sec:structure}

Fix a nonzero shift $k$ and first study the prescribed-shift slice $d_w(k)=2$ before asking whether $k$ is globally minimizing.
When $w$ is compared with $\rot^k w$, each index is paired with the previous point on the directed cycles of $i\mapsto i+k$ on $\Zn$.
Along each cycle, rotation by $k$ is a one-step cyclic shift, so the contribution to $d_w(k)$ is the number of cyclic transitions in the corresponding orbit word.
At distance $2$, this forces exactly one active orbit, and on that orbit the $1$-set is a cyclic interval.

We apply \Cref{lem:shift-cycle-decomp} to the prescribed shift $k$: the cyclic order on each $k$-orbit contributes its boundary count to $d_w(k)$.  We use the following notation throughout this section.
\subsection{Orbit notation and boundary counts}
Fix $n$ and a nonzero shift $k$.  Let $g:=\gcdop(n,k)$ and $m:=n/g$.  The cyclic subgroup $\langle k\rangle\le \Zn$ partitions $\Zn$ into the $g$ orbits
\[
\mathcal{O}_r := \{r + tk : t=0,1,\dots,m-1\}, \qquad r=0,1,\dots,g-1.
\]
For each orbit $\mathcal{O}_r$, define the \emph{orbit word} $a^{(r)}\in\B^m$ by $a^{(r)}_t := w_{r+tk}$.  For any cyclic binary word $x\in\B^m$, write
\[
\partial x:=\{t\in\Zquot{m}: x_t\neq x_{t+1}\},
\]
for its \emph{boundary set}.  Since $(\rot^k w)|_{\mathcal O_r}$ corresponds to the one-step cyclic shift of $a^{(r)}$, \Cref{lem:shift-cycle-decomp} gives
\begin{equation}
\label{eq:transition-count}
d_w(k)=\sum_{r=0}^{g-1}|\partial a^{(r)}|.
\end{equation}
Thus the contribution of a $k$-orbit is exactly the boundary count of its orbit word, and every orbit contribution is even.  Consequently, if $d_w(k)=2$, the decomposition in \Cref{lem:shift-cycle-decomp} realizes the only possible partition of $2$ into nonnegative even orbit contributions,
\[
2=2+0+\cdots+0,
\]
so there is a unique \emph{active} $k$-orbit (see \Cref{lem:active}).

\begin{example}
\label{ex:orbit-decomp}
Take $n=12$ and $k=8$.
Then $g=\gcdop(12,8)=4$ and $m=3$, and the $k$-orbits are the residue classes modulo~$4$:
\[
\mathcal{O}_0=\{0,8,4\},\quad
\mathcal{O}_1=\{1,9,5\},\quad
\mathcal{O}_2=\{2,10,6\},\quad
\mathcal{O}_3=\{3,11,7\}.
\]
Suppose $w$ is constant on $\mathcal{O}_1,\mathcal{O}_2,\mathcal{O}_3$ and, in the orbit order $(0,8,4)$ on $\mathcal{O}_0$, we have
\[
(w_0,w_8,w_4)=(1,1,0).
\]
Thus the orbit word on $\mathcal{O}_0$ is $a^{(0)}=(1,1,0)$, whose boundary set is
\[
\partial a^{(0)}=\{1,2\}.
\]
Its boundary count is therefore $|\partial a^{(0)}|=2$, while the other orbit words are constant and have boundary count $0$.
Hence $d_w(8)=2$ by \Cref{lem:shift-cycle-decomp,eq:transition-count}.

Here $\mathcal{O}_0$ is the unique active orbit, while each of the other $g-1=3$ orbits may be chosen to be all-$0$ or all-$1$ independently.
Once the active orbit word is fixed, these independent choices do not affect $d_w(k)$; this is the basic source of the $2^{g-1}$ factor in prescribed-shift counts and in the generic genuine-interval SP count of \Cref{prop:SP-interval,thm:countSP}.  In the subgroup-shaped exact-minimum classes, the additional primitive quotient--background condition replaces this free-orbit factor by the primitive-word count.
\end{example}

\paragraph{Cyclic intervals.}
For $m\ge 1$, we call a subset $I\subseteq\Zquot{m}$ a \emph{cyclic interval of length $b$} if
\[
I=\{a,a+1,\dots,a+b-1\}\subseteq\Zquot{m},
\]
for some $a\in\Zquot{m}$ and some $b\in\{0,1,\dots,m\}$, where addition is in $\Zquot{m}$.
The cases $b=0$ and $b=m$ mean $I=\varnothing$ and $I=\Zquot{m}$, respectively.
When $1\le b\le m-1$ we call $I$ a \emph{proper} cyclic interval.

\subsection{Active orbits and cyclic intervals}
If $d_w(k)=2$, only one orbit can carry the two boundary points.
We isolate that orbit and the cyclic interval form of its $1$-set.

\begin{lemma}[Exactly one active orbit]
\label{lem:active}
Let $n\ge2$, let $w\in\B^n$, and let $k\in\{1,\dots,n-1\}$.
Write $g:=\gcdop(n,k)$.
If $d_w(k)=2$, then exactly one $k$-orbit contributes a nonzero amount to the sum in \Cref{lem:shift-cycle-decomp}.
In other words, among the $g$ orbits of $\langle k\rangle$, precisely one orbit has internal distance~$2$, and all others have internal distance~$0$.
\end{lemma}

\begin{proof}
For each $k$-orbit $\mathcal O_r$, let $a^{(r)}$ be its orbit word and set
\[
T_r:=|\partial a^{(r)}|.
\]
Then each $T_r$ is a nonnegative even integer.
By \Cref{eq:transition-count},
\[
2=d_w(k)=\sum_{r=0}^{g-1} T_r.
\]
Hence $(T_0,\dots,T_{g-1})$ is a partition of $2$ into nonnegative even integers.
The only such partition is
\[
2=2+0+\cdots+0.
\]
Therefore, exactly one index $r$ satisfies $T_r=2$, while every other index satisfies $T_r=0$.
This is exactly the asserted uniqueness of the active orbit.
\end{proof}

\begin{lemma}[The active-orbit $1$-set is one cyclic interval]
\label{lem:block}
Let $n\ge2$, let $w\in\B^n$, and let $k\in\{1,\dots,n-1\}$.
Write $g:=\gcdop(n,k)$ and $m:=n/g$.
Assume $d_w(k)=2$.
On the unique active $k$-orbit $\mathcal{O}$ from \Cref{lem:active}, the restriction of $w$ has exactly two boundary points.
Equivalently, in the cyclic order of $\mathcal{O}$, the $1$-set forms a single cyclic interval of length $b$ for some $1\le b\le m-1$; that is, after a cyclic reindexing the active orbit word is $0^{m-b}1^b$.
All other $k$-orbits are constant.
\end{lemma}

\begin{proof}
Let $a=(a_0,\dots,a_{m-1})\in\B^m$ be the orbit word on the active orbit $\mathcal O$, written in cyclic order.
By \Cref{lem:active} and \Cref{eq:transition-count},
\[
|\partial a|=2.
\]
Every other $k$-orbit contributes $0$, hence has boundary count $0$ and is therefore constant.
Write
\[
\partial a=\{s,t\}\subseteq\Zquot{m}, \qquad s\neq t.
\]
The two boundary points cut the cycle $\Zquot{m}$ into two nonempty cyclic intervals $A$ and $B$.
Since $a$ changes value only when crossing a point of $\partial a$, it is constant on each of $A$ and $B$.  Because $s,t\in\partial a$, these two constant values are opposite.
After a cyclic reindexing that starts at one boundary point, the orbit word therefore has the form
\[
0^{m-b}1^b \quad \text{or} \quad 1^{m-b}0^b,
\]
for some $1\le b\le m-1$.
Equivalently, the $1$-set is a proper cyclic interval of length $b$.
Conversely, a proper cyclic interval has exactly two boundary points, so its orbit word contributes exactly $2$.
This proves the stated equivalence.
\end{proof}
\begin{remark}[Complement invariance]
\label{rem:complement-invariance}
For a binary word $w=(w_i)_{i\in\Zn}\in\B^n$, write $\overline{w}$ for the binary complement defined by
\[
(\overline{w})_i:=1-w_i\qquad (i\in\Zn).
\]
Then for every $k\in\Zn$,
\[
d_{\overline{w}}(k)=d_w(k),
\]
because $w_i\neq w_{i-k}$ holds if and only if $(1-w_i)\neq (1-w_{i-k})$.
Consequently,
\[
\minDist(\overline{w})=\minDist(w),
\qquad
\Ktwo(\overline{w})=\Ktwo(w).
\]
Thus passing from $w$ to $\overline{w}$ converts a co-singleton $1$-set on the active orbit into a singleton $1$-set without changing the minimizing-shift data.
\end{remark}

\subsection{When do \emph{all} multiples of \texorpdfstring{$k$}{k} have distance~2?}
In class $\DtwoSUB$, every nonzero shift in the subgroup generated by $k$ must have distance $2$; equivalently, if $k$ has order $m$, then $d_w(jk)=2$ for each nonzero residue $j\in\Zquot{m}$.
The local observations used below all come from the same interval-overlap computation, so we record them once in a single statement.

\begin{proposition}[Interval-shift criterion and its stable-block consequences]
\label{prop:interval-shift-criterion}
Let $m\ge 2$, let $I\subseteq\Zquot{m}$ be a proper cyclic interval of length $b\in\{1,\dots,m-1\}$, and let $j\in\{1,\dots,m-1\}$.
Set
\[
\beta:=\min\{b,m-b\},
\qquad
\delta:=\min\{j,m-j\}.
\]
Then
\[
|I\triangle(I+j)|=2\min\{\beta,\delta\}.
\]
Consequently:
\begin{enumerate}[label=(\alph*),leftmargin=2.5em]
\item
\[
|I\triangle(I+j)|=2
\quad\Longleftrightarrow\quad
b\in\{1,m-1\}\ \text{or}\ j\equiv \pm1\pmod m;
\]
\item if $m\in\{2,3\}$, then $|I\triangle(I+j)|=2$ for every nonzero $j$;
\item if $m\ge 4$ and $2\le b\le m-2$, then $|I\triangle(I+j)|=2$ if and only if $j\equiv \pm1\pmod m$;
\item if $w\in\B^n$ satisfies $d_w(k)=2$, and if the active $k$-orbit from \Cref{lem:block} has size $m=n/\gcdop(n,k)$ and interval length $b$, then
\[
d_w(jk)=2\ \text{for every }j=1,\dots,m-1
\quad\Longleftrightarrow\quad
b\in\{1,m-1\};
\]
\item under the same hypotheses, if $d_w(\ell k)=2$ for some $\ell\in\{1,\dots,m-1\}$ with $\ell\not\equiv \pm1\pmod m$, then $b\in\{1,m-1\}$.
\end{enumerate}
In particular, among non-half-turn subgroup-shaped cases whose active-orbit $1$-set is a singleton or co-singleton, quotient order $3$ is exactly the case in which the nonzero part of the subgroup is one antipodal minimizing pair.
\end{proposition}

\begin{proof}
The quantity $|I\triangle(I+j)|$ is unchanged by translating $I$, by replacing $I$ with its complement (since $I^c\triangle(I^c+j)=I\triangle(I+j)$), and by replacing $j$ with $-j$ (equivalently, with $m-j$). Hence we may assume
\[
I=\{0,1,\dots,\beta-1\},
\qquad
1\le \beta\le m/2,
\qquad
1\le \delta\le m/2,
\qquad
j=\delta.
\]
Then $\beta+\delta\le m$, so $I+\delta$ does not wrap around $0$ and
\[
I\cap(I+\delta)=
\begin{cases}
\{\delta,\delta+1,\dots,\beta-1\},&\delta<\beta,\\
\varnothing,&\delta\ge \beta
\end{cases}.
\]
Hence $|I\cap(I+\delta)|=\max\{\beta-\delta,0\}$, and therefore
\[
|I\triangle(I+j)|=2\bigl(\beta-|I\cap(I+\delta)|\bigr)=2\min\{\beta,\delta\}.
\]
This proves part~(a); parts~(b) and~(c) are immediate consequences.
For parts~(d) and~(e), let $I$ be the $1$-set on the active $k$-orbit.
By \Cref{lem:block}, $I$ is a cyclic interval and every other $k$-orbit is constant, so
\[
d_w(jk)=|I\triangle(I+j)|\qquad (1\le j\le m-1).
\]
Part~(d) now restates part~(a) uniformly over all nonzero $j$, and part~(e) is the contrapositive of part~(c).
\end{proof}

The next proposition isolates the local subgroup-recovery step for the singleton/co-singleton branch.  The active $k$-orbit is first normalized by placing its unique exceptional point at the origin.  The quotient--background word then records the background values and their period subgroup.  The notation used in the proposition has the following local roles.
\begin{description}[leftmargin=3.8em,labelwidth=2.8em,labelsep=.6em]
\item[$x$] the original word satisfying $d_x(k)=2$;
\item[$w$] the rotated single-defect representative obtained from $x$ after the exceptional point is placed at index $0$;
\item[$c$] the quotient--background word in $\B^g$ on the $g=\gcdop(n,k)$ residue classes;
\item[$v$] the orbit-constant lift determined by $c$;
\item[$P$] the period subgroup $P=\rho\Zquot{g}$ of $c$;
\item[$\rho$] the subgroup divisor recovered from the period subgroup of $c$;
\item[$c_0$] the bit distinguishing the singleton and co-singleton $1$-set forms on the active orbit;
\item[$w(c)$] the single-defect word reconstructed from $c$ by the displayed normal-form formulas.
\end{description}
With this notation, the proposition gives the full minimizing-shift set $\Ktwo(w)$.  It also recovers the bit $c_0$.  The supplementary material contains a worked quotient--background example illustrating the same notation.

This proposition is the local step connecting the prescribed-shift analysis to the defect--background reconstruction in \Cref{sec:defect-background}.  That section uses it to identify the normalized representative, recover subgroup data from the quotient order, and pass to the necklace-level bijection used in \Cref{sec:enum}; for that reason the local recovery is proved here, before the global reconstruction is assembled.

\begin{proposition}[Single-defect subgroup recovery from the quotient--background word]
\label{prop:defect-quotient}
Let $x\in\B^n$ and let $k\in\{1,\dots,n-1\}$ satisfy $d_x(k)=2$.
Write
\[
g:=\gcdop(n,k),\qquad m:=n/g,
\]
and assume $m\ge 3$ and that the active $k$-orbit from \Cref{lem:active,lem:block} has interval length $b\in\{1,m-1\}$.
Rotate $x$ so that the unique exceptional point on that active orbit lies at index $0$, and call the rotated word $w$.
Equivalently, $w$ is in single-defect normal form relative to $k$: there is a quotient--background word $c=(c_0,\dots,c_{g-1})\in\B^g$ and its orbit-constant lift $v\in\B^n$ such that
\[
w_0=1-c_0,
\qquad
w_{tk}=c_0\ \ (1\le t\le m-1),
\qquad
w_{r+tk}=c_r\ \ (r\ne 0),
\]
while
\[
v_{r+tk}=c_r\qquad (0\le r\le g-1,\ t\in\Zquot{m}).
\]
Here $c_0$ is the common value on the nondefect points $tk$ with $1\le t\le m-1$ of the active $k$-orbit. Thus $c_0=0$ when the $1$-set on that orbit is a singleton ($b=1$), and $c_0=1$ when the $1$-set on that orbit is a co-singleton ($b=m-1$).
For use below, let $w(c)$ denote the word defined by the preceding formulas for $w_0$, $w_{tk}$, and $w_{r+tk}$.
For $\ell\in\Zn$, write uniquely
\[
\ell\equiv e+ak\pmod n,
\]
with $e\in\{0,1,\dots,g-1\}$ and $a\in\Zquot{m}$, and let
\[
P:=\{e\in\Zquot{g}:\rot^e c=c\},
\]
be the period subgroup of the quotient--background word.
Then:
\begin{enumerate}[label=(\alph*),leftmargin=2.5em]
\item for every $\ell\in\Zn$,
\[
d_v(\ell)=m\,d_c(e);
\]
\item for every nonzero $\ell\in\Zn$,
\[
d_w(\ell)=2
\quad\Longleftrightarrow\quad
e\in P;
\]
\item if $\rho\mid g$ is the unique divisor with $P=\rho\Zquot{g}$, then
\[
\Ktwo(w)=\rho\Zn\setminus\{0\}.
\]
\item the bit $c_0$ distinguishes the singleton and co-singleton $1$-set forms on the active orbit:
\[
c_0=0 \iff w_0=1 \text{ and } w_{tk}=0\ \ (1\le t\le m-1),
\]
so the $1$-set on the active $k$-orbit is a singleton, while
\[
c_0=1 \iff w_0=0 \text{ and } w_{tk}=1\ \ (1\le t\le m-1),
\]
so the $1$-set on the active $k$-orbit is a co-singleton.
\end{enumerate}
In particular, once a word is put in single-defect normal form relative to a minimizing shift, the quotient--background word and its period subgroup determine the full minimizing-shift set $\Ktwo(w)$ exactly; the counting results below only refine this same data by quotient order and primitivity.
\end{proposition}

\begin{proof}
Fix $\ell\in\Zn$ and write $\ell\equiv e+ak\pmod n$.
Because $v$ is constant on each $k$-orbit with orbit value $c_r$, shifting by $\ell$ carries $\mathcal O_r$ to $\mathcal O_{r-e}$; hence every index on $\mathcal O_r$ compares $c_r$ with $c_{r-e}$.
Thus the whole orbit contributes either $0$ or $m$ mismatches, and summing over $r$ gives
\[
d_v(\ell)=m\,\bigl|\{r\in\Zquot{g}:c_r\ne c_{r-e}\}\bigr|=m\,d_c(e),
\]
proving part~(a).

Assume now that $\ell\ne 0$.
If $e\in P$, then $\rot^e c=c$, so part~(a) gives $d_v(\ell)=0$, that is, $\rot^\ell v=v$.
Since $w$ differs from $v$ only at $0$, while $\rot^\ell w$ differs from $v$ only at $\ell$, the words $w$ and $\rot^\ell w$ differ exactly at those two positions; hence $d_w(\ell)=2$.
If instead $e\notin P$, then $d_c(e)$ is a positive even integer by \Cref{lem:dw-even}, so $d_c(e)\ge 2$ and part~(a) yields $d_v(\ell)\ge 2m$.
Since $w$ and $v$ differ only at the defect position $0$, we have
\[
\Ham(w,v)=1,
\qquad
\Ham(\rot^\ell w,\rot^\ell v)=\Ham(w,v)=1.
\]
Applying the triangle inequality for Hamming distance to the pairs $(w,\rot^\ell w)$ and $(v,\rot^\ell v)$ gives
\[
\bigl|\Ham(w,\rot^\ell w)-\Ham(v,\rot^\ell v)\bigr|
   \le \Ham(w,v)+\Ham(\rot^\ell w,\rot^\ell v)=2.
\]
Therefore
\[
d_w(\ell)=\Ham(w,\rot^\ell w)\ge \Ham(v,\rot^\ell v)-2=d_v(\ell)-2\ge 2m-2\ge 4,
\]
because $m\ge 3$.
Thus $d_w(\ell)=2$ holds exactly when $e\in P$, proving part~(b).

For part~(c),
\[
\Ktwo(w)=\{\ell\in\Zn\setminus\{0\}:\ell\bmod g\in P\}.
\]
If $P=\rho\Zquot{g}$, its full preimage in $\Zn$ is $\rho\Zn$, so $\Ktwo(w)=\rho\Zn\setminus\{0\}$.
Part~(d) is immediate from the defining formulas
\[
w_0=1-c_0,
\qquad
w_{tk}=c_0\ \ (1\le t\le m-1),
\]
in \Cref{prop:defect-quotient}.
\end{proof}

\subsection{Normal forms relative to a fixed minimizing shift}
Fixing one minimizing shift $k$ leaves only two local branches.  In the genuine-interval branch, the full minimizer set is the single antipodal pair $\{\pm k\}$.  In the single-defect normal-form branch of \Cref{prop:defect-quotient}, the quotient--background period subgroup recovers the full minimizer set and records the singleton/co-singleton $1$-set form on the active orbit.

\begin{proposition}[Normal forms from one minimizing shift]
\label{prop:normalforms}
Let $w\in\B^n$ satisfy $\minDist(w)=2$, and fix $k\in\Ktwo(w)$.
Write
\[
g:=\gcdop(n,k),\qquad m:=n/g,
\]
so the subgroup generated by $k$ is $\langle k\rangle=g\Zn$ and has order $m$.
Let $\mathcal O_0,\dots,\mathcal O_{g-1}$ be the $k$-orbits, chosen so that $\mathcal O_0$ is the unique active orbit from \Cref{lem:active}.
Let $b\in\{1,\dots,m-1\}$ be the length of the $1$-set on $\mathcal O_0$, viewed as the cyclic interval from \Cref{lem:block}.
Then the following hold.
\begin{enumerate}[label=(\roman*),leftmargin=2.5em]
\item If $2\le b\le m-2$, then
\[
\Ktwo(w)=\{k,n-k\}.
\]
\item If $b\in\{1,m-1\}$ and $m\ge 3$, then there exists a unique divisor $\tau\mid g$ such that
\[
\Ktwo(w)=\tau\Zn\setminus\{0\}.
\]
Equivalently, after rotating $w$ so that the unique exceptional point on the active $k$-orbit lies at index $0$, the resulting representative is in the single-defect normal form of \Cref{prop:defect-quotient}.  In that normal form, $c_0=0$ records that the $1$-set on the active orbit is a singleton ($b=1$), and $c_0=1$ records that it is a co-singleton ($b=m-1$).
\item If $m=2$, then $k=n/2$ and this is the half-turn case.
\end{enumerate}
\end{proposition}

\begin{proof}
If $m=2$, then $k$ is the unique element of order $2$ in $\Zn$, namely the half-turn $n/2$.
Assume $m\ge 3$.
Choose a residue representative $r_0\in\{0,\dots,g-1\}$ for the active $k$-orbit, and relabel the $k$-orbits by
\[
\mathcal O_s=\{r_0+s+tk:t\in\Zquot{m}\}\qquad(s\in\Zquot{g}),
\]
so that $\mathcal O_0$ is the active orbit.
By \Cref{lem:block}, the active orbit has a $1$-set that is a cyclic interval of length $b$, and every other $k$-orbit is constant.

If $2\le b\le m-2$, then for shifts inside $\langle k\rangle$ the only possibilities with distance $2$ are $\pm k$ by \Cref{prop:interval-shift-criterion}.
If $\ell\notin\langle k\rangle$, write $r:=\ell\bmod g\in\{1,\dots,g-1\}$, so $\mathcal O_r\neq\mathcal O_0$ and $\mathcal O_{g-r}\neq\mathcal O_0$.
Subtracting $\ell$ sends each orbit $\mathcal O_s$ to $\mathcal O_{s-r}$; in particular, every index of $\mathcal O_0$ is paired with an index of the constant orbit $\mathcal O_{g-r}$, and every index of $\mathcal O_r$ is paired with an index of $\mathcal O_0$.
Now $\mathcal O_0$ carries the interval word of length $b$, while $\mathcal O_{g-r}$ and $\mathcal O_r$ are constant with values $c_{g-r},c_r\in\B$.
Hence the $\mathcal O_0$ contribution is
\[
\#\{t\in\Zquot{m}: w_{r_0+tk}\neq c_{g-r}\}
=
\begin{cases}
b,& c_{g-r}=0,\\
m-b,& c_{g-r}=1
\end{cases},
\]
and the $\mathcal O_r$ contribution is
\[
\#\{t\in\Zquot{m}: c_r\neq w_{r_0+tk}\}
=
\begin{cases}
b,& c_r=0,\\
m-b,& c_r=1
\end{cases}.
\]
Because $2\le b\le m-2$, both orbit contributions are at least $2$, so $d_w(\ell)\ge 4$.
Thus no such $\ell$ lies in $\Ktwo(w)$, and therefore $\Ktwo(w)=\{k,n-k\}$.

If $b\in\{1,m-1\}$, rotate $w$ so that the unique exceptional point on the active orbit lies at $0\in\mathcal O_0$, and call the rotated word $x$.
Then $x$ is in the single-defect normal form of \Cref{prop:defect-quotient}, which yields a unique divisor $\tau\mid g$ with
\[
\Ktwo(x)=\tau\Zn\setminus\{0\}.
\]
Because rotation preserves the minimizing-shift set, $\Ktwo(w)=\Ktwo(x)$, so
\[
\Ktwo(w)=\tau\Zn\setminus\{0\}.
\]
This is exactly the subgroup-shaped alternative in part~(ii).
\end{proof}

The next lemma isolates the primitivity of the normalized representatives used below in the half-turn, interval, and single-defect counts.

\begin{lemma}[Primitive normalized representatives]
\label{lem:normalized-primitive}
The normalized representatives used below in the HT, genuine-interval SP, and single-defect branches are primitive. More precisely:
\begin{enumerate}[label=(\alph*),leftmargin=2.5em]
\item if $n=2m$, $w_0=1$, $w_m=0$, and $w_j=w_{j+m}$ for $1\le j\le m-1$, then $w$ is primitive;
\item let $k\in\{1,\dots,n-1\}$ have order $m\ge 4$. If the unique active $k$-orbit is
\[
O_0:=\{0,k,\dots,(m-1)k\},
\]
its $1$-set on $O_0$ is the interval $\{0,1,\dots,b-1\}\subseteq\Zquot{m}$ with $2\le b\le m-2$, and every other $k$-orbit is constant, then $w$ is primitive;
\item if $w$ is in the single-defect normal form of \Cref{prop:defect-quotient} relative to a shift $k$, with $m:=n/\gcdop(n,k)\ge 3$, then $w$ is primitive.
\end{enumerate}
\end{lemma}

\begin{proof}
For part~(a), suppose $\rot^s w=w$ for some nonzero $s\in\Zn$. Then
\[
w_s=w_0=1,
\qquad
w_{s+m}=w_m=0,
\]
so the opposite pair $\{s,s+m\}$ is unequal. By hypothesis $\{0,m\}$ is the unique unequal opposite pair, hence $\{s,s+m\}=\{0,m\}$. Therefore, $s=0$ or $s=m$. But $s=m$ would force $w_m=w_0$, contrary to $w_m=0$ and $w_0=1$. Thus no such nonzero $s$ exists.

For part~(b), suppose $\rot^s w=w$ for some nonzero $s\in\Zn$. Because $O_0$ is the unique nonconstant $k$-orbit, the translate $O_0+s$ is also nonconstant, so uniqueness forces $O_0+s=O_0$. Hence $s=jk$ for some $j\in\{1,\dots,m-1\}$. On $O_0$ the $1$-set is the interval $I:=\{0,1,\dots,b-1\}\subseteq\Zquot{m}$, while every other $k$-orbit is constant, so
\[
d_w(jk)=|I\triangle(I+j)|.
\]
The interval-shift criterion in \Cref{prop:interval-shift-criterion} gives $|I\triangle(I+j)|>0$, since $I$ is a proper cyclic interval and $j\ne 0$ in $\Zquot{m}$. Hence $d_w(jk)>0$, contradicting $\rot^s w=w$.

For part~(c), suppose $\rot^s w=w$ for some nonzero $s\in\Zn$. If $s\notin\langle k\rangle$, then rotation by $s$ carries the unique active $k$-orbit of the single-defect normal form to a different $k$-orbit, which would then also be nonconstant, a contradiction. Thus $s=tk$ for some $t\in\{1,\dots,m-1\}$. By the defining formulas in \Cref{prop:defect-quotient},
\[
w_{tk}=c_0,
\qquad
w_0=1-c_0.
\]
But $\rot^s w=w$ gives $w_{tk}=w_0$, which is impossible. Hence $w$ is primitive.
\end{proof}

\section{Minimizer-set compatibility and the distance-\texorpdfstring{$2$}{2} collapse}
\label{sec:trichotomy}

We next record how different prescribed-shift witnesses interact.
The finite-difference notation below is a compact compatibility calculus for minimizer sets.
The subsection on cyclic derivatives and support overlap is not used as a dependency in the proof of \Cref{thm:trichotomy}; it is included to give a compact consistency check for simultaneous minimizers before the collapse theorem.
For distance $2$, each minimizing support has only two points; the collapse proof then uses the prescribed-shift normal-form dichotomy of \Cref{prop:normalforms} together with the subgroup arithmetic of \Cref{lem:subgroup-arith}.
\subsection{Cyclic derivatives and support overlap}
For $w\in\B^n$ and $k\in\Zn$, define the binary cyclic derivative
\[
\Delta_k w:=w+\rot^k w\pmod 2.
\]
Thus $(\Delta_k w)_i=1$ exactly when $w_i\ne w_{i-k}$, and therefore
\[
d_w(k)=|\operatorname{supp}(\Delta_k w)|.
\]

\begin{proposition}[Cocycle identity for rotational mismatches]
\label{prop:delta-cocycle}
Let $n\ge2$ and let $w\in\B^n$.  For all $k,\ell\in\Zn$,
\[
\Delta_{k+\ell}w=\Delta_k w+\rot^k(\Delta_\ell w)\pmod 2.
\]
Consequently, if $B_k:=\operatorname{supp}(\Delta_k w)$, then
\[
B_{k+\ell}=B_k\triangle \rot^k B_\ell,
\qquad
 d_w(k+\ell)=d_w(k)+d_w(\ell)-2|B_k\cap \rot^kB_\ell|.
\]
\end{proposition}

\begin{proof}
At index $i$,
\[
(\Delta_k w)_i+(\rot^k\Delta_\ell w)_i
=(w_i+w_{i-k})+(w_{i-k}+w_{i-k-\ell})
=w_i+w_{i-k-\ell}
=(\Delta_{k+\ell}w)_i,
\]
over $\mathbb F_2$.
Taking supports converts addition over $\mathbb F_2$ into symmetric difference, and the displayed cardinality formula for $d_w(k+\ell)$ follows from $|A\triangle B|=|A|+|B|-2|A\cap B|$.
\end{proof}

\begin{lemma}[Support-overlap bound for distance-$2$ minimizers]
\label{lem:minimizer-overlap-bound}
Let $w\in\B^n$ satisfy $\minDist(w)=2$.
For $k\in\Ktwo(w)$ write $B_k:=\operatorname{supp}(\Delta_k w)$.
If $k,\ell\in\Ktwo(w)$ and $k+\ell\ne0$ in $\Zn$, then
\[
|B_k\cap\rot^k B_\ell|\le 1.
\]
Moreover,
\[
k+\ell\in\Ktwo(w)
\quad\Longleftrightarrow\quad
|B_k\cap\rot^kB_\ell|=1.
\]
\end{lemma}

\begin{proof}
Since $k,\ell\in\Ktwo(w)$, we have $|B_k|=|B_\ell|=2$.
The cardinality formula in \Cref{prop:delta-cocycle} gives
\[
d_w(k+\ell)=4-2|B_k\cap\rot^kB_\ell|.
\]
Because $k+\ell\ne0$, the definition of $\minDist(w)=2$ gives $d_w(k+\ell)\ge2$, hence
$|B_k\cap\rot^kB_\ell|\le 1$.
Equality holds exactly when $d_w(k+\ell)=2$, that is, exactly when $k+\ell\in\Ktwo(w)$.
\end{proof}

\begin{remark}[Distance-$2$ boundary supports]
When $\minDist(w)=2$, each minimizing derivative support $B_k$ has size $2$.
The cocycle identity therefore gives a compact way to check compatibility among simultaneous minimizers.
In this paper the complete distance-$2$ classifier is proved in the original orbit language by \Cref{prop:normalforms} and \Cref{lem:subgroup-arith}; the support-overlap statement above gives a consistency check for simultaneous minimizers.
The derivative calculation above is a consistency check rather than a dependency for the proof of \Cref{thm:trichotomy}; the proof begins in the next subsection.
\end{remark}

\subsection{\texorpdfstring{The distance-$2$ collapse theorem}{The distance-2 collapse theorem}}

We now classify binary necklaces $\mathcal N$ with $\minDist(\mathcal N)=2$ by the shape of $\Ktwo(\mathcal N)$, using the representative-independent convention from \Cref{rem:necklace-word-invariants}.
The proof has two inputs. First, \Cref{prop:normalforms} says that any non-half-turn minimizing shift belongs to either the genuine-interval branch or the single-defect normal-form branch. In the genuine-interval branch, it gives the unique antipodal minimizing pair; in the single-defect branch, the whole minimizer set is the nonzero part of one subgroup of $\Zn$. Second, \Cref{lem:subgroup-arith} determines the HT/SP/SUB label from the order of that subgroup. These inputs leave no residual multi-pair case at distance $2$.

\begin{lemma}[Negation symmetry of $\Ktwo(w)$]
\label{lem:K2sym}
Let $w\in\B^n$ with $\minDist(w)=2$.
Then $k\in\Ktwo(w)$ if and only if $n-k\in\Ktwo(w)$.
In particular, $\Ktwo(w)$ is a disjoint union of antipodal pairs $\{k,n-k\}$, except that when $n$ is even the half-turn $n/2$ may appear as a singleton.
\end{lemma}

\begin{proof}
We always have $d_w(k)=d_w(n-k)$ by the symmetry noted after~\eqref{eq:dwdef}.
Thus $d_w(k)=2$ if and only if $d_w(n-k)=2$, which is exactly the claim.
\end{proof}

\begin{lemma}[Arithmetic of subgroup-shaped minimizer sets]
\label{lem:subgroup-arith}
Let $\tau\mid n$, set
\[
K:=\tau\Zn\setminus\{0\},
\qquad
M:=n/\tau.
\]
Then:
\begin{enumerate}[label=(\alph*),leftmargin=2.5em]
\item $|K|=M-1$;
\item if $n$ is even, then
\[
n/2\in K \iff M\text{ is even};
\]
\item
\[
|K|=2 \iff M=3.
\]
\end{enumerate}
\end{lemma}

\begin{proof}
The subgroup $\tau\Zn\le \Zn$ has exactly $M=n/\tau$ elements, so removing $0$ leaves $|K|=M-1$.

If $n$ is even, then $n/2\in K$ exactly when $n/2$ is a nonzero multiple of $\tau$.
Writing $n=\tau M$, this is equivalent to $M/2\in\mathbb Z$, that is, to $M$ being even.

Finally, $|K|=2$ holds if and only if $M-1=2$, that is, if and only if $M=3$.
\end{proof}

\begin{theorem}[Distance-$2$ minimizer-set collapse]
\label{thm:trichotomy}
Fix $n\ge 2$ and let $\mathcal N$ be a binary necklace of length $n$ with $\minDist(\mathcal N)=2$, interpreted via \Cref{rem:necklace-word-invariants}.
Set
\[
K:=\Ktwo(\mathcal N).
\]
Then exactly one of the following holds:
\begin{enumerate}[label=(\roman*),leftmargin=2.5em]
\item \textbf{HT case ($\DtwoHT$; \Cref{def:DtwoHT}):}
$n$ is even and $n/2\in K$.
In this case, if $\tau:=\gcdop(n,K)$ and $M:=n/\tau$, then
\[
K=\tau\Zn\setminus\{0\},
\]
and $M$ is even; equivalently, $\mathcal N\in\DtwoHT(n;\tau)$.

\item \textbf{SP case ($\DtwoSP$; \Cref{def:DtwoSP}):}
$|K|=2$.
Equivalently,
\[
K=\{k,n-k\},
\]
for a unique $1\le k<n/2$, so $\mathcal N\in\DtwoSP(n)$.

\item \textbf{SUB case ($\DtwoSUB$; \Cref{def:DtwoSUB}):}
$|K|>2$ and, when $n$ is even, $n/2\notin K$.
Then there exists a unique divisor $\tau\mid n$ such that
\[
K=\tau\Zn\setminus\{0\}.
\]
Writing $M:=n/\tau$, we have $M$ odd and $M\ge 5$; equivalently, $\mathcal N\in\DtwoSUB(n;\tau)$.
\end{enumerate}
Consequently,
\[
\DtwoMP(n)=\varnothing,
\qquad
\Dtwo(n)=\DtwoHT(n)\;\dot\cup\;\DtwoSP(n)\;\dot\cup\;\DtwoSUB(n).
\]
Thus the residual multi-pair class is empty at distance $2$.
\end{theorem}

\begin{proof}
Choose a representative $w\in\mathcal N$, so that $K=\Ktwo(w)$ by \Cref{rem:necklace-word-invariants}.
If $n$ is even and $K=\{n/2\}$, then
\[
K=(n/2)\Zn\setminus\{0\},
\]
so this is the pure HT case.
Otherwise, if the half-turn is present, choose $k\in K\setminus\{n/2\}$; if it is not present, choose any $k\in K$.
Apply \Cref{prop:normalforms} to this $k$.
The half-turn alternative of \Cref{prop:normalforms} is excluded by the choice of $k$.
If $n$ is even and $n/2\in K$, the genuine-interval branch is also impossible: it would give $K=\{k,n-k\}$, and neither element is $n/2$ because $k\ne n/2$.
Thus any non-pure HT necklace is forced into the subgroup-shaped branch below.

If this application of \Cref{prop:normalforms} falls in the genuine-interval branch, then part~(i) gives
\[
K=\{k,n-k\},
\]
which is exactly the SP case.
If the $1$-set on the active $k$-orbit is a singleton or co-singleton, \Cref{prop:normalforms}(ii) gives a divisor $\tau\mid n$ such that
\[
K=\tau\Zn\setminus\{0\}.
\]
Taking gcds with $n$ shows that this divisor is $\gcdop(n,K)$, hence it is unique.
With $M:=n/\tau$, \Cref{lem:subgroup-arith} gives the three possible labels directly:
\[
M\text{ even}\iff n/2\in K,
\qquad
M=3\iff |K|=2,
\qquad
M\text{ odd and }M\ge5\iff |K|>2\text{ and }n/2\notin K.
\]
Thus the subgroup-shaped branch contributes precisely the mixed half-turn cases, the order-$3$ SP boundary case, and the SUB cases.
Together with the pure HT boundary and the genuine-interval SP branch, these alternatives are disjoint and exhaustive.
The residual multi-pair branch in \Cref{def:d2-taxonomy,def:DtwoMP} therefore has no distance-$2$ instances, proving $\DtwoMP(n)=\varnothing$ and the stated partition of $\Dtwo(n)$.
\end{proof}

\begin{remark}[Interpreting the trichotomy]
\label{rem:trichotomy-reading}
The HT class consists of the pure boundary $K=\{n/2\}$ together with the mixed half-turn subgroup-shaped case.  The SP class consists of the genuine-interval branch with a single antipodal pair together with the order-$3$ SP boundary case.  Thus the distance-$2$ collapse leaves no residual MP class: the prescribed-shift normal forms allow only a single antipodal pair or a full subgroup-shaped minimizer set.
\end{remark}

\noindent Consequently, whenever $\mathcal N\in\DtwoHT(n;\tau)$ or $\mathcal N\in\DtwoSUB(n;\tau)$, the preliminary notation above satisfies
\[
\tau=\tautwo(\mathcal N)=\gcdop(n,\Ktwo(\mathcal N)).
\]
In the HT case the quotient $M:=n/\tau$ is even, while in the SUB case $\tau$ is the strict period and $M$ is odd (hence $M\ge 5$).

\begin{remark}[Computational classifier]
The preceding proof gives the case split used in the manuscript: $\DtwoMP(n)=\varnothing$ and the partition is exhaustive.  The supplementary material records a direct finite-profile classifier for examples and the finite check.
\end{remark}

\section{Branch normal forms and defect--background reconstruction}
\label{sec:defect-background}

After the minimizer-set class has been identified, words are reconstructed within each class.  At distance $2$, the only nontrivial normal form beyond the genuine-interval branch is the subgroup-shaped defect--background model: a periodic background with one exceptional point.
For these subgroup-shaped cases, \Cref{prop:defect-quotient,prop:defect-branch-normalization,cor:defect-primitive-readout,thm:defect-primitive-branch-split} recover $\Ktwo(w)$, fix the normalized representative, determine the HT/SP/SUB class, and give the converse necklace-level bijection.
\Cref{cor:defect-counting-dictionary} records the cardinality consequence used in the divisor-sum formulas of \Cref{sec:enum}.

\paragraph{Dependency guide for the subgroup-shaped branch.}
The reconstruction results used in this section feed into the enumeration as follows.
\begin{center}
\small
\begin{tabular}{p{0.26\textwidth}p{0.35\textwidth}p{0.23\textwidth}}
\hline
Result & Role & Enumeration use \\
\hline
\Cref{prop:defect-quotient} & Recovers the single-defect quotient--background model and the subgroup-shaped minimizer set from a minimizing shift. & Set-up for all subgroup-shaped HT/SP/SUB slices. \\
\Cref{prop:defect-branch-normalization} & Fixes the unique normalized representative with exceptional point at index $0$. & Prevents overcounting before primitive readout. \\
\Cref{cor:defect-primitive-readout} & Identifies the primitive quotient--background word and the divisor $\tau$ from the normalized representative. & Parameterizes the mixed HT slices, the order-$3$ SP boundary case, and the SUB slices. \\
\Cref{thm:defect-primitive-branch-split} & Gives the converse necklace-level bijection and separates the $M=3$, even-$M$, and odd-$M$ cases. & Assigns the subgroup-shaped slices to $\DtwoSP$, $\DtwoHT$, and $\DtwoSUB$. \\
\Cref{cor:defect-counting-dictionary} & Converts the bijection into the primitive-word cardinality for each slice. & Supplies the defect--background terms in \Cref{thm:countHT-tau,thm:countSP,thm:countSUB}. \\
\hline
\end{tabular}
\end{center}

The subgroup-shaped part of the distance-$2$ classification is described by a periodic background with one exceptional point.  The following moment identity is the structural test for subgroup-shaped minimizer sets.  At distance $2$, it forces exactly one nonconstant quotient orbit, which is then reconstructed by the defect--background normal form.

\begin{proposition}[Subgroup-shaped moment identity at distance $2$]
\label{prop:subgroup-moment-identity}
Let $[w]\in\Dtwo(n)$, and assume that its minimizer set has the subgroup-shaped form
\[
\Ktwo(w)=\tau\Zn\setminus\{0\}.
\]
Write $n=\tau M$ and decompose $w$ into the $\tau$ orbit words $u^{(r)}\in\B^M$ on the residue classes modulo $\tau$.
If $a_r:=|\{j\in\Zquot{M}:u^{(r)}_j=1\}|$, then
\begin{equation}
\label{eq:subgroup-moment-identity}
\sum_{r=0}^{\tau-1}a_r(M-a_r)=M-1.
\end{equation}
Consequently, exactly one $\tau$-orbit word is nonconstant, and that word has weight $1$ or $M-1$.
\end{proposition}

\begin{proof}
Since $[w]\in\Dtwo(n)$, the set $\Ktwo(w)$ is nonempty.  Hence the assumed set $\tau\Zn\setminus\{0\}$ is nonempty, so $M=n/\tau\ge2$.
For $1\le j\le M-1$, the shift $j\tau$ acts within each $\tau$-orbit, so
\[
d_w(j\tau)=\sum_{r=0}^{\tau-1}d_{u^{(r)}}(j).
\]
Since every nonzero multiple of $\tau$ is a minimizer, the left side is $2$ for each $j=1,\dots,M-1$.
Summing over these $j$ gives
\[
2(M-1)=\sum_{r=0}^{\tau-1}\sum_{j=1}^{M-1}d_{u^{(r)}}(j).
\]
For a binary word $u$ of length $M$ and weight $a$, the standard symmetric-difference count gives
\[
\sum_{j=0}^{M-1}d_u(j)=2a(M-a),
\]
and the $j=0$ term is zero.
Applying this identity to each $u^{(r)}$ and dividing by $2$ gives \eqref{eq:subgroup-moment-identity}.
The least positive value of $a(M-a)$ is $M-1$, attained only at $a=1$ or $a=M-1$; hence the displayed sum has exactly one nonzero summand, and that summand has the asserted weight.
\end{proof}

\Cref{prop:subgroup-moment-identity} explains why the condition $\Ktwo(w)=\tau\Zn\setminus\{0\}$ leaves one nonconstant quotient orbit.  The necklace-level bijection and counts below use \Cref{prop:defect-quotient,prop:defect-branch-normalization,cor:defect-primitive-readout,thm:defect-primitive-branch-split,cor:defect-counting-dictionary}.

The final distance-$2$ enumeration uses five disjoint counting pieces refining the three nonempty classes:
\begin{enumerate}[label=(\roman*),leftmargin=2.5em]
\item pure HT: $M=2$; parameter $\tau=n/2$; counted in \Cref{thm:countHT,cor:countHT-pure};
\item order-$3$ SP boundary: $M=3$; primitive quotient--background words of length $\tau=n/3$; counted in \Cref{thm:countSP};
\item mixed half-turn boundary cases: $M=n/\tau$ even, $M\ge 4$; primitive quotient--background words of length $\tau$; counted in \Cref{thm:countHT-tau};
\item generic genuine-interval SP: order $m=\operatorname{ord}_{\Zn}(k)\ge 4$ and interval data $(b,c_1,\dots,c_{g-1})$; interval branch \Cref{prop:SP-interval}; counted in \Cref{thm:countSP};
\item odd-quotient SUB: $M=n/\tau$ odd, $M\ge 5$; primitive quotient--background words of length $\tau$; counted in \Cref{thm:countSUB}.
\end{enumerate}
These five pieces are disjoint and exhaustive for $\Dtwo(n)$ and retain the HT/SP/SUB class assignment.
The rest of this section records the subgroup-shaped branch in three steps: uniqueness of the normalized representative, primitive determination of the subgroup data, and the converse bijection.
The actual cardinality formulas are collected in \Cref{sec:enum}.

\begin{proposition}[Unique normalized representative for the shift $k=\tau$]
\label{prop:defect-branch-normalization}
Let $\tau\mid n$ and write $M:=n/\tau$.
Assume that either $M=3$, or $M$ is even and $M\ge 4$, or $M$ is odd and $M\ge 5$.
Let $x\in\B^n$ satisfy $\minDist(x)=2$ and $d_x(\tau)=2$.
For the shift $k=\tau$, let $\mathcal O$ be the unique active $\tau$-orbit from \Cref{lem:active,lem:block}.
Assume that the $1$-set on $\mathcal O$ is a cyclic interval of length $1$ or $M-1$.
In the notation of \Cref{prop:defect-quotient} for the shift $k=\tau$,
\[
g=\gcdop(n,k)=\tau,
\qquad
m=n/g=M.
\]
Rotate $x$ so that the unique exceptional point on the active $\tau$-orbit lies at index $0$, and call the resulting representative $w$.
The representative $w$ is in the single-defect normal form of \Cref{prop:defect-quotient} relative to $k=\tau$.
Hence there is a quotient--background word $c=(c_0,\dots,c_{\tau-1})\in\B^\tau$ such that
\[
w_0=1-c_0,\qquad w_{t\tau}=c_0\ \ (1\le t\le M-1),\qquad w_{r+t\tau}=c_r\ \ (r\ne 0).
\]
The displayed formulas define the word $w(c)$, and in this situation $w(c)=w$.
For the fixed minimizing shift $k=\tau$ and defect position $0$, the normalized representative $w$ is unique.
Consequently, the associated quotient--background word $c$ is unique as well.
\end{proposition}

\begin{proof}
Existence of the displayed normal form is exactly \Cref{prop:defect-quotient} with the chosen minimizing shift $k=\tau$. For uniqueness, let $w'$ be another representative of the same necklace in the same normal form. Then $w'=\rot^s w$ for some $s\in\Zn$.

If $s\notin\tau\Zn$, the unique active $\tau$-orbit of $w$ is carried to a different $\tau$-orbit, contradicting the normalization that places that orbit through $0$. If $s\in\tau\Zn\setminus\{0\}$, the active orbit is preserved but the unique defect moves away from index $0$, again contradicting the normalization. Hence $s=0$, so $w'=w$, and the defining formulas force the same quotient--background word.
\end{proof}

\begin{corollary}[Primitive determination of the subgroup-shaped branch]
\label{cor:defect-primitive-readout}
Let $\tau\mid n$ and write $M:=n/\tau$.
Assume that either $M=3$, or $M$ is even and $M\ge 4$, or $M$ is odd and $M\ge 5$.
Let $c=(c_0,\dots,c_{\tau-1})\in\B^\tau$.
Let $w=w(c)$ be the word defined by the single-defect formulas in \Cref{prop:defect-branch-normalization} relative to $k=\tau$.
Write
\[
P:=\{e\in\Zquot{\tau}:\rot^e c=c\}=\rho\Zquot{\tau},
\]
for the period subgroup of the quotient--background word. Then:
\begin{enumerate}[label=(\alph*),leftmargin=2.5em]
\item
\[
\Ktwo(w)=\rho\Zn\setminus\{0\},
\]
so
\[
\Ktwo(w)=\tau\Zn\setminus\{0\}
\quad\Longleftrightarrow\quad
P=\{0\}
\quad\Longleftrightarrow\quad
c\text{ is primitive};
\]
\item the bit $c_0$ distinguishes the singleton and co-singleton $1$-set forms on the active orbit:
\[
c_0=0 \iff w_0=1 \text{ and } w_{t\tau}=0\ \ (1\le t\le M-1),
\]
so the $1$-set on the active $\tau$-orbit is a singleton, while
\[
c_0=1 \iff w_0=0 \text{ and } w_{t\tau}=1\ \ (1\le t\le M-1),
\]
so the $1$-set on the active $\tau$-orbit is a co-singleton;
\item if $c$ is primitive, then the necklace class $[w]$ lies in $\Dtwo(n)$, $\tau$ is the least positive minimizer in $\Ktwo([w])$, and the quotient-order split is
\[
\begin{aligned}
M=3 &\Longleftrightarrow [w]\in\DtwoSP(n;3),\\
M\ \text{is even and}\ M\ge 4 &\Longleftrightarrow [w]\in\DtwoHT(n;\tau),\\
M\ \text{is odd and}\ M\ge 5 &\Longleftrightarrow [w]\in\DtwoSUB(n;\tau)
\end{aligned}.
\]
In the odd-quotient case, the strict period is $\tau$.
\end{enumerate}
\end{corollary}

\begin{proof}
By construction, $d_w(\tau)=2$, and the active $\tau$-orbit has interval length $1$ or $M-1$.
All other $\tau$-orbits are constant.
Thus \Cref{prop:defect-quotient} applies with $k=\tau$ and $g=\tau$.
Parts~(a) and~(b) are exactly \Cref{prop:defect-quotient}(c,d) in this notation.
The displayed equivalence in part~(a) uses that a length-$\tau$ word is primitive exactly when its period subgroup in $\Zquot{\tau}$ is trivial.

Assume now that $c$ is primitive.
Then part~(a) gives $\Ktwo(w)=\tau\Zn\setminus\{0\}$; in particular, $d_w(\tau)=2$.
By \Cref{lem:normalized-primitive}(c), the normalized representative $w$ is primitive.
Since all nontrivial rotational Hamming distances are even by \Cref{lem:dw-even}, we obtain $\minDist(w)=2$, so the necklace of $w$ lies in $\Dtwo(n)$.
The subgroup formula also makes $\tau$ the least positive minimizer.
The quotient-order alternatives then follow from \Cref{lem:subgroup-arith}.
The alternatives $M=3$, even $M\ge 4$, and odd $M\ge 5$ correspond respectively to the order-$3$ SP boundary case, the mixed half-turn case, and the SUB case.
This is precisely the trichotomy of \Cref{thm:trichotomy}; in the SUB case the strict period is $\tau$ by definition.
\end{proof}

\begin{remark}[Primitivity as a canonical divisor condition]
\label{rem:quotient-primitivity-canonical}
The primitivity condition on the quotient--background word is a canonical-divisor condition.  In the notation of \Cref{cor:defect-primitive-readout}, if the quotient--background word has a nontrivial period subgroup $P=\rho\Zquot{\tau}$ with $\rho<\tau$, then the same single-defect representative has
\[
\Ktwo(w)=\rho\Zn\setminus\{0\}.
\]
Thus a nonprimitive quotient--background word does not define a new slice indexed by $\tau$; it belongs to the subgroup-shaped slice indexed by the smaller divisor $\rho$.  Requiring $c$ to be primitive, equivalently $P=\{0\}$, is therefore exactly what makes the divisor $\tau$ canonical and prevents overlap among the subgroup-shaped slices.
\end{remark}

\begin{theorem}[Converse reconstruction and bijection for subgroup-shaped slices]
\label{thm:defect-primitive-branch-split}
Let $\tau\mid n$ and write $M:=n/\tau$.
Assume that either $M=3$, or $M$ is even and $M\ge 4$, or $M$ is odd and $M\ge 5$.
Then:
\begin{enumerate}[label=(\alph*),leftmargin=2.5em]
\item if $\mathcal N$ is a necklace class with $\Ktwo(\mathcal N)=\tau\Zn\setminus\{0\}$, then $\mathcal N$ lies in the subgroup-shaped defect--background branch of \Cref{prop:normalforms}.
Moreover, there is a unique normalized representative $w=w(c)\in\mathcal N$ in the single-defect normal form of \Cref{prop:defect-quotient} relative to $k=\tau$.
The associated quotient--background word $c\in\B^\tau$ is unique and primitive;
\item the assignment is
\[
c\longmapsto [w(c)].
\]
This assignment is a bijection from primitive quotient--background words of length $\tau$ onto the following subgroup-shaped slice of order $M$:
\[
\begin{cases}
\DtwoSP(n;3), & M=3,\\
\DtwoHT(n;\tau), & M\ \text{is even and}\ M\ge 4,\\
\DtwoSUB(n;\tau), & M\ \text{is odd and}\ M\ge 5
\end{cases}.
\]
\end{enumerate}
After the chosen minimizing shift $k=\tau$ and the defect position $0$ are fixed, necklace equivalence imposes no additional quotient in these bijective cases.
\end{theorem}

\begin{proof}
For part~(a), choose any representative $x$ of $\mathcal N$.
By \Cref{rem:necklace-word-invariants}, $\Ktwo(x)=\tau\Zn\setminus\{0\}$.
In particular, $d_x(\tau)=2$.
We first verify that $x$ is in the distance-$2$ regime.
Applying \Cref{lem:active,lem:block} to the prescribed shift $\tau$ gives a unique active $\tau$-orbit, whose $1$-set is a proper cyclic interval $I\subseteq\Zquot{M}$.
Every other $\tau$-orbit is constant.
Suppose that $\rot^s x=x$ for some nonzero $s\in\Zn$.
Then translation by $s$ must carry the unique active $\tau$-orbit to itself.
Hence $s=j\tau$ with $1\le j\le M-1$.
The equality $\rot^s x=x$ would force $I+j=I$.
However, the interval-shift formula in \Cref{prop:interval-shift-criterion} gives $|I\triangle(I+j)|>0$ for every proper cyclic interval and every nonzero $j$.
This contradiction shows that $x$ is primitive and has no nontrivial zero-distance rotation.
Since $d_x(\tau)=2$ and all nontrivial rotational Hamming distances are even by \Cref{lem:dw-even}, we have $\minDist(x)=2$.
Therefore, \Cref{prop:normalforms} applies with the minimizing shift $k=\tau$.

Because $M\ge 3$, the pure half-turn alternative is excluded.
If $M=3$, the genuine-interval branch is impossible because it would require an interval length $b$ with $2\le b\le M-2=1$.
If $M\ge 4$, the genuine-interval branch would force $\Ktwo(x)=\{\tau,n-\tau\}$, but $\tau\Zn\setminus\{0\}$ has size $M-1\ge 3$.
Hence only the subgroup-shaped alternative remains.
On the unique active $\tau$-orbit, \Cref{lem:block} therefore gives a $1$-set that is a cyclic interval of length $b\in\{1,M-1\}$.
Rotate $x$ so that the unique exceptional point on that active orbit lies at index $0$, and call the rotated word $w$.
Let $c_0$ be the common value of the remaining points $t\tau$ on that orbit for $1\le t\le M-1$.
Equivalently, $c_0=0$ when the $1$-set on the active orbit is a singleton, and $c_0=1$ when it is a co-singleton.
For each other $\tau$-orbit, let $c_r$ be its constant value. Then
\[
w_0=1-c_0,\qquad w_{t\tau}=c_0\ \ (1\le t\le M-1),\qquad w_{r+t\tau}=c_r\ \ (r\ne 0).
\]
Thus $w$ is in the single-defect normal form of \Cref{prop:defect-quotient} relative to $k=\tau$.
The uniqueness of $w$ and $c$ is exactly \Cref{prop:defect-branch-normalization}.
By \Cref{cor:defect-primitive-readout}(a), $c$ is primitive because $\Ktwo(w)=\tau\Zn\setminus\{0\}$.

For part~(b), let $c$ be a primitive quotient--background word of length $\tau$.
Define $w(c)$ by the defining formulas in \Cref{prop:defect-branch-normalization}.
By construction, $w(c)$ is in the single-defect normal form relative to $k=\tau$.
Therefore, \Cref{cor:defect-primitive-readout}(c) places $[w(c)]$ in the stated order-$M$ slice.
The uniqueness statement in \Cref{prop:defect-branch-normalization} gives the injectivity of $c\mapsto [w(c)]$.
Surjectivity follows from part~(a), which reconstructs the unique normalized representative from any necklace in that subgroup-shaped slice.
\end{proof}

\begin{remark}[Why primitive words are counted here]
\label{rem:defect-words-not-necklaces}
The bijection in \Cref{thm:defect-primitive-branch-split} is formulated for primitive quotient--background words, not for primitive quotient--background necklaces. Once the chosen minimizing shift $k=\tau$ is fixed and the exceptional point on the unique active $\tau$-orbit is placed at index $0$, rotations do not give an additional quotient. A rotation that changes the active orbit or moves the exceptional point violates the normalization; a rotation that preserves both is trivial by \Cref{prop:defect-branch-normalization}. Thus two primitive words $c,c'\in\B^\tau$ with $[w(c)]=[w(c')]$ are already equal as words.

The edge case $\tau=1$ is included. The primitive length-$1$ words $0$ and $1$ give the singleton and co-singleton normalized representatives. Because complements are not identified in binary necklace equivalence, they remain two contributions. The divisor sums below therefore use the primitive-word count
\[
\PrimCt{\tau}=\sum_{d\mid\tau}\mu(d)2^{\tau/d},
\]
directly, with no division by $\tau$ or by $2$.
\end{remark}

The next corollary records only the counting consequence of \Cref{cor:defect-primitive-readout} and \Cref{thm:defect-primitive-branch-split}. The normalization in \Cref{rem:defect-words-not-necklaces} explains why the primitive quotient--background words are counted directly.

\begin{corollary}[Counting consequence for the subgroup-shaped defect branch]
\label{cor:defect-counting-dictionary}
Let $\tau\mid n$ and write $M:=n/\tau$.
\begin{enumerate}[label=(\alph*),leftmargin=2.5em]
\item If $M=2$, then the subgroup-shaped normal-form boundary consists exactly of the pure half-turn classes with minimizer set $\{n/2\}$;
\item If $M=3$, or if $M$ is even and $M\ge 4$, or if $M$ is odd and $M\ge 5$, then the necklace-level bijection of \Cref{thm:defect-primitive-branch-split}(b) identifies the corresponding order-$M$ subgroup-shaped slice with the primitive quotient--background words of length $\tau$.  Hence that slice has cardinality
\[
\PrimCt{\tau}=\sum_{d\mid\tau}\mu(d)2^{\tau/d}.
\]
Equivalently, this count applies to the order-$3$ SP boundary case when $M=3$, the mixed half-turn slice when $M$ is even and $M\ge 4$, and the SUB slice when $M$ is odd and $M\ge 5$.
\item In the cases covered by part~(b), $\tau$ is the least positive minimizer and, when $M$ is odd and $M\ge 5$, the strict period.
\end{enumerate}
\end{corollary}

\begin{proof}
Part~(a) is the remaining boundary case.  If $M=2$, then $\tau=n/2$ and the nonzero part $\tau\Zn\setminus\{0\}$ of the subgroup $\tau\Zn$ is exactly $\{n/2\}$.

For part~(b), the bijection is \Cref{thm:defect-primitive-branch-split}(b).  The cardinality assertion follows from \Cref{rem:defect-words-not-necklaces} and the primitive-word count \eqref{eq:prim-words}.

Part~(c) is the least-minimizer and odd-quotient strict-period statement from \Cref{cor:defect-primitive-readout}(c).
\end{proof}

\section{Enumeration}
\label{sec:enum}

This section turns the recovered normal-form branches into formulas.
It counts the HT, SP, and SUB classes identified above; the MP class is empty by \Cref{thm:trichotomy}.
It also keeps the prescribed-shift comparison separate from the exact-minimum class counts, the primitive-word normalization, and the final necklace division via \Cref{lem:positive-minimum-primitive}.
The pure HT boundary $M=2$ is treated in \Cref{thm:countHT,cor:countHT-pure}.
The subgroup-shaped slices are counted through the defect--background dictionary of \Cref{cor:defect-counting-dictionary}, and the SP pieces from the genuine-interval branch are counted in \Cref{prop:SP-interval}.
These ingredients give the class formulas and the aggregate total in \Cref{cor:countD2}.
They also give the two equivalent forms of the main theorem: the expanded M\"obius--totient aggregate and the compact $\PrimCt{\cdot}$ form.
The refined HT slices remain in the main text to distinguish the mixed half-turn subgroup-shaped contribution from the pure half-turn boundary.
Section S5.1 of the supplementary material gives a small-length finite check of the same formulas by direct orbit enumeration, including the deterministic scan specification, sample records, field specification, aggregate count table, and aggregate output summary.

In particular, throughout this section a primitive quotient--background word of length $\tau$ means one whose minimal period, in the sense fixed before \eqref{eq:prim-words}, is exactly $\tau$.
The normalization in \Cref{rem:defect-words-not-necklaces} explains why these primitive quotient--background words are counted as words, not as necklace or complement classes, once the chosen minimizing shift and defect position have been fixed.
The indicator convention stated in the introduction is in force throughout this enumeration section: an expression such as $\One_{3\mid n}\sum_{d\mid(n/3)}\mu(d)2^{n/(3d)}$ is read casewise, so the inner divisor sum is evaluated only when $3\mid n$ and the whole term is otherwise zero.

\subsection{Counting \texorpdfstring{$\DtwoHT$}{distance-2 HT} (half-turn class)}

\begin{theorem}[HT summand in \Cref{thm:main}]
\label{thm:countHT}
For $n\ge 2$,
\[
|\DtwoHT(n)| = \One_{2\mid n}\,2^{\frac{n}{2}-1}.
\]
\end{theorem}

\begin{proof}
Assume $n$ is even and write $n=2m$.
A word satisfies $d_w(m)=2$ exactly when among the $m$ opposite pairs $\{i,i+m\}$ there is exactly one unequal pair, because mismatches under the half-turn always occur in opposite pairs.

For $[w]\in\DtwoHT(n)$, rotate this unique unequal pair to $\{0,m\}$ and choose the orientation with $(w_0,w_m)=(1,0)$.
Then each necklace has a unique representative satisfying
\[
w_0=1,\qquad w_m=0,\qquad w_j=w_{j+m}\ \text{for }1\le j\le m-1.
\]
Thus the HT count is taken at the level of these normalized representatives: every HT necklace contributes exactly one such word, and every such word determines exactly one necklace.
Conversely, any choice of the $m-1$ remaining bits produces such a normalized representative, so there are $2^{m-1}$ of them.
By \Cref{lem:normalized-primitive}(a), every such representative is primitive, hence $\minDist(w)\neq 0$.
Together with $d_w(m)=2$, which gives $\minDist(w)\le 2$, \Cref{lem:dw-even} therefore yields $\minDist(w)=2$.
\end{proof}

\begin{theorem}[Refined half-turn subgroup slices]
\label{thm:countHT-tau}
Let $n\ge 2$ be even and let $\tau\mid (n/2)$.
Write $m:=n/\tau$.
If $m\ge 4$, then
\[
|\DtwoHT(n;\tau)|=\sum_{d\mid \tau}\mu(d)\,2^{\tau/d}.
\]
\end{theorem}

\begin{proof}
Fix such a divisor $\tau$ and write $m:=n/\tau$, so $m$ is even. The hypothesis $m\ge 4$ places us in the mixed half-turn slice rather than the pure boundary case $\tau=n/2$.
By part~(b) of \Cref{cor:defect-counting-dictionary}, the mixed half-turn slice $\DtwoHT(n;\tau)$ is in necklace-level bijection with the primitive quotient--background words of length $\tau$.
By \eqref{eq:prim-words}, the number of such words is $\sum_{d\mid \tau}\mu(d)\,2^{\tau/d}$.
\end{proof}

\begin{corollary}[Pure half-turn slice]
\label{cor:countHT-pure}
Let $n=2m\ge 2$.
Then the number of necklace classes $\mathcal N\in\DtwoHT(n)$ with $\Ktwo(\mathcal N)=\{m\}$ equals
\[
|\DtwoHT(n;n/2)|=\sum_{d\mid m}\mu(d)\,2^{m/d}-2^{m-1}.
\]
\end{corollary}

\begin{proof}
By \Cref{thm:trichotomy}, every necklace class $\mathcal N\in\DtwoHT(n)$ lies in a unique refined slice $\DtwoHT(n;\tau)$ with $\tau\mid m$. The proper divisors $\tau<m$ are exactly the mixed half-turn slices counted by \Cref{thm:countHT-tau}, while $\tau=m=n/2$ is the pure boundary case $\Ktwo(\mathcal N)=\{m\}$. Therefore
\[
|\DtwoHT(n)|=|\DtwoHT(n;n/2)|+\sum_{\substack{\tau\mid m\\ \tau<m}} |\DtwoHT(n;\tau)|.
\]
For each proper divisor $\tau<m$, the corresponding mixed half-turn slice is counted in \Cref{thm:countHT-tau}, so
\[
|\DtwoHT(n;\tau)|=\sum_{d\mid \tau}\mu(d)\,2^{\tau/d}.
\]

Every binary word of length $m$ has a unique minimal period $\tau\mid m$.  If $\tau$ is that minimal period, then the word is obtained by repeating the unique primitive word given by its first $\tau$ letters exactly $m/\tau$ times.  Conversely, repeating a primitive word of length $\tau$ exactly $m/\tau$ times gives a length-$m$ word with minimal period $\tau$.
Hence the primitive words of length $\tau$ parameterize the proper-divisor refined slice $\DtwoHT(n;\tau)$ for each $\tau<m$, and counting length-$m$ words by minimal period gives
\[
2^m=\sum_{\tau\mid m}\#\{\text{primitive words of length }\tau\}.
\]
Using \eqref{eq:prim-words} and isolating the term $\tau=m$ therefore yields
\[
\sum_{\substack{\tau\mid m\\ \tau<m}} |\DtwoHT(n;\tau)|
=\sum_{\substack{\tau\mid m\\ \tau<m}}\sum_{d\mid \tau}\mu(d)\,2^{\tau/d}
=2^m-\sum_{d\mid m}\mu(d)\,2^{m/d}.
\]
Subtracting the total contribution of these proper-divisor refined slices from $|\DtwoHT(n)|=2^{m-1}$ (from \Cref{thm:countHT}) leaves precisely the pure half-turn slice $\DtwoHT(n;n/2)$, which is the claimed formula.
\end{proof}

\begin{remark}[Pure versus mixed half-turn]
\Cref{cor:countHT-pure} isolates the pure boundary $\Ktwo(\mathcal N)=\{n/2\}$ from the mixed half-turn subgroup-shaped slices $\DtwoHT(n;\tau)$ with $\tau<n/2$. This separation is used again when the aggregate formula is assembled.
\end{remark}

\subsection{Counting \texorpdfstring{$\DtwoSP$}{distance-2 SP} (single-pair class)}

Here the minimizer set is a single antipodal pair $\{k,n-k\}$, so the natural refinement parameter is the order $m=\operatorname{ord}_{\Zn}(k)$ from \Cref{def:DtwoSPByOrder}.

\begin{proposition}[Interval model for the generic SP contribution]
\label{prop:SP-interval}
Let $m\ge 4$ divide $n$, set $g:=n/m$, and fix a representative $k\in\Zn$ of an unordered pair $\{\pm k\}=\{k,n-k\}$ whose elements have order $m$ in the sense of \Cref{def:DtwoSPByOrder}.
For each necklace class $\mathcal N$ with $\Ktwo(\mathcal N)=\{k,n-k\}$, there is a unique rotated representative $w\in\mathcal N$ for which the unique active $k$-orbit is
\[
O_0:=\{0,k,\dots,(m-1)k\},
\]
and the cyclic interval forming the $1$-set on $O_0$ starts at $0$.
With this normalization, there is a unique $b\in\{2,\dots,m-2\}$ such that
\[
\{t\in\Zquot{m}:w_{tk}=1\}=\{0,1,\dots,b-1\}.
\]
For every $1\le r\le g-1$, the orbit $O_r:=\{r+tk:t\in\Zquot{m}\}$ is constant; denote its value by $c_r\in\B$.
The assignment is
\[
w\longmapsto (b,c_1,\dots,c_{g-1}).
\]
This assignment is a bijection from the set of such normalized representatives to $\{2,\dots,m-2\}\times\B^{g-1}$.
Fix an unordered minimizer pair $\{\pm k\}=\{k,n-k\}$ of order $m$ and choose one representative $k$ for the normalization.  The bijection above gives exactly $(m-3)2^{g-1}$ necklace classes $\mathcal N$ satisfying $\Ktwo(\mathcal N)=\{k,n-k\}$.
\end{proposition}

\begin{proof}
Let $w$ satisfy $\Ktwo(w)=\{k,n-k\}$.
By \Cref{lem:active,lem:block}, there is a unique active $k$-orbit and its $1$-set is a proper cyclic interval.
Since $|\Ktwo(w)|=2$, \Cref{prop:interval-shift-criterion} excludes the extremal lengths $1$ and $m-1$, so the interval length satisfies $b\in\{2,\dots,m-2\}$.
Rotate so that this active orbit is $O_0$ and the cyclic interval forming its $1$-set starts at $0$.
The normalization is unique: a rotation outside $\langle k\rangle$ moves the active orbit, while a nonzero multiple of $k$ preserves $O_0$ but shifts the chosen starting point of the $1$-set away from $0$.
Thus, for the chosen representative $k$ of the unordered pair $\{\pm k\}$, counting the tuples $(b,c_1,\dots,c_{g-1})$ is already counting necklaces with that prescribed minimizer pair via their unique normalized representatives.
Replacing the representative by $-k$ only reverses the interval coordinate, so it does not produce a second contribution and does not change the pair-based count.
All other $k$-orbits are constant by \Cref{lem:active,lem:block}, so each necklace yields a unique tuple $(b,c_1,\dots,c_{g-1})$.

Conversely, any tuple in $\{2,\dots,m-2\}\times\B^{g-1}$ defines a word with a single nonconstant $k$-orbit, namely the interval $\{0,\dots,b-1\}$ on $O_0$.
That word satisfies $d_w(k)=d_w(n-k)=2$.
By \Cref{lem:normalized-primitive}(b), the constructed word is primitive; together with $d_w(k)=2$, parity gives $\minDist(w)=2$.
The active $k$-orbit has interval length $2\le b\le m-2$, so \Cref{prop:normalforms}(i) gives
\[
\Ktwo(w)=\{k,n-k\}.
\]
Thus the construction is inverse to the normalization map.
The count is therefore
\[
|\{2,\dots,m-2\}|\,|\B^{g-1}|=(m-3)\,2^{g-1}.
\]
\end{proof}

By part~(b) of \Cref{cor:defect-counting-dictionary}, the subgroup-shaped defect--background contribution to SP is only the order-$3$ SP boundary case, while every SP slice with $m\ge 4$ is counted by the genuine-interval branch of \Cref{prop:SP-interval}.

\begin{theorem}[SP summand in \Cref{thm:main}]
\label{thm:countSP}
Let $n\ge 2$.
\begin{enumerate}[label=(\alph*),leftmargin=2.5em]
\item For $m=3$,
\[
|\DtwoSP(n;3)|=\One_{3\mid n}\sum_{d\mid (n/3)} \mu(d)\,2^{\frac{n}{3d}}.
\]
\item For each divisor $m\ge 4$ of $n$,
\[
|\DtwoSP(n;m)|=\frac{\varphi(m)}{2}\,2^{\frac{n}{m}-1}\,(m-3).
\]
\item Consequently,
\[
|\DtwoSP(n)|
=
\One_{3\mid n}\sum_{d\mid (n/3)} \mu(d)\,2^{\frac{n}{3d}}
\;+
\sum_{\substack{m\mid n\\ m\ge 4}}
\frac{\varphi(m)}{2}\,2^{\frac{n}{m}-1}\,(m-3).
\]
\end{enumerate}
\end{theorem}

\begin{proof}
For each divisor $m\mid n$, part~(b) of \Cref{cor:defect-counting-dictionary} contributes only the order-$3$ SP boundary case, while \Cref{prop:SP-interval} counts the SP pieces in the genuine-interval branch for order $m\ge 4$. The interval count is attached to the unordered minimizer pair $\{\pm k\}$, not to the two orientations separately, and shifts of order $m$ come in $\varphi(m)/2$ such unordered pairs.
If $3\nmid n$, the convention in \Cref{def:DtwoSPByOrder} gives $\DtwoSP(n;3)=\varnothing$, and the right-hand side of part~(a) is zero.
If $3\mid n$, the order-$3$ SP boundary case is unique. Writing $\tau=n/3$, part~(b) of \Cref{cor:defect-counting-dictionary} identifies $\DtwoSP(n;3)$ with primitive quotient--background words of length $\tau$, so \eqref{eq:prim-words} gives part~(a).
If $m\ge 4$, \Cref{prop:SP-interval} already counts one normalized interval representative for each necklace with prescribed unordered minimizer pair $\{\pm k\}$, so each such pair contributes $(m-3)2^{n/m-1}$ necklaces, which gives part~(b).
Summing the contribution of the order-$3$ SP boundary case together with all divisors $m\ge 4$ yields part~(c).
\end{proof}

\subsection{Counting \texorpdfstring{$\DtwoSUB$}{distance-2 SUB} (subgroup-pattern class)}

Again part~(b) of \Cref{cor:defect-counting-dictionary} reduces the odd-quotient SUB count to primitive quotient--background words of length $\tau$.

\begin{theorem}[SUB summand in \Cref{thm:main}]
\label{thm:countSUB}
Let $n\ge 2$.
For each divisor $\tau\mid n$ such that $m:=n/\tau$ is odd and $m\ge 5$,
\[
|\DtwoSUB(n;\tau)|=\sum_{d\mid \tau}\mu(d)\,2^{\tau/d}.
\]
Consequently,
\[
|\DtwoSUB(n)|=\sum_{\substack{\tau\mid n\\ (n/\tau)\ \mathrm{odd}\\ n/\tau\ge 5}}
\sum_{d\mid \tau}\mu(d)\,2^{\tau/d}.
\]
\end{theorem}

\begin{proof}
Fix an admissible divisor $\tau$.  By \Cref{cor:defect-counting-dictionary}(b), the slice $\DtwoSUB(n;\tau)$ is counted by primitive words in the quotient--background model of length $\tau$.
The primitive-word formula~\eqref{eq:prim-words} counts those words as $\sum_{d\mid \tau}\mu(d)\,2^{\tau/d}$, giving the first formula.
Summing over all admissible $\tau$ gives the second formula.
\end{proof}

\begin{remark}[Excluded small odd quotients]
In \Cref{thm:countSUB} we assume that $m:=n/\tau$ is odd and at least $5$, so that the nonzero part $\tau\Zn\setminus\{0\}$ of the subgroup $\tau\Zn$ has size $m-1>2$.
The remaining odd quotients $m\in\{1,3\}$ are excluded for exactly this reason: when $m=1$ this nonzero part is empty, and when $m=3$ it has size $2$ and the resulting necklaces fall into the $\DtwoSP(n;3)$ slice (counted in \Cref{thm:countSP}).
\end{remark}

\subsection{Counting the union \texorpdfstring{$\Dtwo$}{D2}}

\begin{corollary}[Aggregate count in \Cref{thm:main}]
\label{cor:countD2}
For every $n\ge 2$,
\[
|\Dtwo(n)|=|\DtwoHT(n)|+|\DtwoSP(n)|+|\DtwoSUB(n)|.
\]
Equivalently, this sum has the explicit M\"obius--totient form
\[
|\Dtwo(n)|
=\One_{2\mid n}\,2^{\frac n2-1}
+\One_{3\mid n}\sum_{d\mid (n/3)} \mu(d)\,2^{\frac{n}{3d}}
+\sum_{\substack{m\mid n\\ m\ge 4}}\frac{\varphi(m)}{2}\,2^{\frac{n}{m}-1}\,(m-3)
+\sum_{\substack{\tau\mid n\\ (n/\tau)\ \mathrm{odd}\\ n/\tau\ge 5}}\sum_{d\mid \tau}\mu(d)\,2^{\tau/d}.
\]
Using the primitive-word shorthand, this is equivalently
\[
|\Dtwo(n)|
=\One_{2\mid n}\,2^{\frac n2-1}
+\One_{3\mid n}\PrimCt{n/3}
+\sum_{\substack{m\mid n\\ m\ge 4}}\frac{\varphi(m)}{2}\,2^{\frac{n}{m}-1}\,(m-3)
+\sum_{\substack{\tau\mid n\\ (n/\tau)\ \mathrm{odd}\\ n/\tau\ge 5}}\PrimCt{\tau}.
\]
\end{corollary}

\begin{proof}
By \Cref{thm:trichotomy}, $\Dtwo(n)$ is the disjoint union of its HT, SP, and SUB subclasses, with the residual multi-pair class empty.
Apply \Cref{thm:countHT,thm:countSP,thm:countSUB} and add.
\end{proof}

\begin{proof}[Proof of \Cref{thm:main}]
For $n=1$, the convention is
\[
\Dtwo(1)=\DtwoHT(1)=\DtwoSP(1)=\DtwoSUB(1)=\DtwoMP(1)=\varnothing.
\]
It gives the stated disjoint decomposition and the empty residual class.  The same convention gives zero for all displayed class formulas under the indicator convention.  Let $n\ge2$.  \Cref{thm:trichotomy} gives the disjoint HT/SP/SUB decomposition and the emptiness of $\DtwoMP(n)$.  The three class formulas are \Cref{thm:countHT,thm:countSP,thm:countSUB}.  Adding these disjoint summands gives the compact aggregate formula.  Substituting $\PrimCt{\ell}=\sum_{d\mid \ell}\mu(d)2^{\ell/d}$ gives the displayed M\"obius--totient divisor-sum form.  This is the form recorded in \Cref{cor:countD2}.
\end{proof}

\begin{remark}[Interpretation of the aggregate formula]
\Cref{cor:countD2} is exactly the disjoint HT/SP/SUB total obtained by summing the HT, SP, and SUB class formulas.  Its role here is structural.  The corollary combines the classwise decomposition into one closed M\"obius--totient divisor-sum expression, and the adjacent $\PrimCt{\ell}$ shorthand records the primitive-word source of the order-$3$ SP boundary case and SUB summands.
\end{remark}

The supplementary material gives a uniform divisor-sum formulation of \Cref{cor:countD2}, a Lambert-series identity, and asymptotic consequences.

\subsection{Odd prime length: structure and count}
\label{sec:prime-structure}

For odd primes, every nonzero shift generates the whole $p$-cycle, so the general one-orbit interval normal form applies verbatim. For $u\in\Zquot{p}\setminus\{0\}$ and $1\le t\le p-1$, set
\[
I_{u,t}:=\{0,u,\dots,(t-1)u\}\subseteq \Zquot{p},
\qquad
w_{u,t}:=\mathbf 1_{I_{u,t}}.
\]

\begin{proposition}[Arithmetic intervals at odd prime length]
\label{prop:prime-interval-normal-form}
Let $p$ be an odd prime, let $w\in\B^p$, and let $u\in\Zquot{p}\setminus\{0\}$.
If $\Ham(w,\rot^u w)=2$, then $\minDist(w)=2$. Moreover, for this fixed generator $u$, there exist $s\in\Zquot{p}$ and a unique $t\in\{1,\dots,p-1\}$ such that
\[
\rot^s w=w_{u,t}.
\]
Conversely, every word $w_{u,t}$ satisfies $\Ham(w_{u,t},\rot^u w_{u,t})=2$.
Moreover,
\[
[w_{u,t}]=[w_{-u,t}],
\qquad
\Ktwo(w_{u,t})=
\begin{cases}
\Zquot{p}\setminus\{0\}, & t\in\{1,p-1\},\\
\{u,-u\}, & 2\le t\le p-2
\end{cases}.
\]
\end{proposition}

\begin{proof}
Assume first that $\Ham(w,\rot^u w)=2$.
By \Cref{lem:dw-even}, every rotational Hamming distance is even, so $\minDist(w)\le 2$.
If $\minDist(w)=0$, then $w=\rot^v w$ for some nonzero $v\in\Zquot{p}$.
Since $p$ is prime, every nonzero $v$ generates $\Zquot{p}$, so $w$ is constant on the whole $p$-cycle and hence constant.
But then $\Ham(w,\rot^u w)=0$, contradicting the assumption.
Therefore, $\minDist(w)\neq 0$, and since it is even we get $\minDist(w)=2$.
Since $u$ generates $\Zquot{p}$, \Cref{prop:normalforms} applies with $k=u$ and gives a single active orbit on the whole $p$-cycle.
Hence, for the fixed generator $u$, there exist $s\in\Zquot{p}$ and a unique $t\in\{1,\dots,p-1\}$ such that the $1$-set of $\rot^s w$ is $I_{u,t}$; equivalently, $\rot^s w=w_{u,t}$.
Conversely, \Cref{prop:interval-shift-criterion} gives $\Ham(w_{u,t},\rot^u w_{u,t})=2$ for every $1\le t\le p-1$.
Also $I_{-u,t}=-(t-1)u+I_{u,t}$, hence $[w_{u,t}]=[w_{-u,t}]$.
The formula for $\Ktwo(w_{u,t})$ is then the prime case of \Cref{prop:interval-shift-criterion}.
\end{proof}

\begin{corollary}[Odd-prime specialization of \Cref{thm:main}]
\label{cor:prime-interval-count}
Let $p$ be an odd prime.
Then
\[
|\Dtwo(p)|=\frac{(p-1)(p-3)}{2}+2=\frac{p^{2}-4p+7}{2}.
\]
\end{corollary}

\begin{proof}
Specializing \Cref{cor:countD2} to an odd prime $p$ gives the total directly from the global HT/SP/SUB count formulas, so there is no HT contribution. Part~(b) of \Cref{cor:defect-counting-dictionary} gives the subgroup-shaped contribution directly at necklace level, so no further quotient remains to be taken.
If $p=3$, the only defect--background divisor is $\tau=1$, which has quotient order $M=3$, so part~(b) of \Cref{cor:defect-counting-dictionary} gives exactly the contribution $2$ from the order-$3$ SP boundary case, while the generic order-$p$ interval term and the SUB sum both vanish.
If $p\ge 5$, the divisor $\tau=1$ has odd quotient order $M=p\ge 5$, so part~(b) of \Cref{cor:defect-counting-dictionary} gives exactly the SUB term $2$, while the generic SP contribution is the order-$p$ interval term $\frac{(p-1)(p-3)}{2}$.
In both cases the total is $\frac{(p-1)(p-3)}{2}+2$.
\end{proof}

\begin{remark}[Prime specialization]
\Cref{prop:prime-interval-normal-form,cor:prime-interval-count} summarize the prime case inside the necklace-level classifier: every distance-$2$ necklace of odd prime length is an arithmetic interval.  The order-$3$ SP boundary case belongs to SP, and for each prime $p\ge 5$, the prime-length cases split between SP and SUB.
\end{remark}

\section{Conclusion}
\label{sec:discussion}

The classification above shows that distance $2$ is rigid: one prescribed minimizing shift either gives a proper cyclic interval on its active orbit, producing a single antipodal minimizer pair, or gives a singleton/co-singleton defect over a periodic background, producing the nonzero part of a subgroup.  This dichotomy rules out the residual multi-pair case and yields the disjoint HT/SP/SUB enumeration in \Cref{thm:main}.

The result separates three levels that are easily conflated: prescribed-shift counts, exact-minimum necklaces, and classwise minimizer-set structure.  The final formulas are necklace-level formulas, and the passage from words to necklaces is justified by the equivalence of positive $\minDist$ and primitivity.  The odd-prime specialization gives the compact count $|\Dtwo(p)|=(p^2-4p+7)/2$ and the arithmetic-interval normal form.

\paragraph{Outlook.}
The natural continuation is the exact-distance enumeration problem for higher even distances: for a fixed even $\Delta\ge4$, determine the number of necklaces with $\minDist=\Delta$ and describe the possible shapes of the full minimizer set
\[
\{k\in\{1,\ldots,n-1\}:d_w(k)=\Delta\}.
\]
The distance-$2$ case suggests two cautions for this broader problem.  First, prescribed-shift formulas should be treated as projections rather than as substitutes for exact-minimum necklace enumeration.  Second, structural alternatives should be converted into an explicit priority classifier before summing, because natural minimizer-set properties may overlap.  We do not claim here that the HT/SP/SUB priority pattern persists at larger distances; the open problem is to find the appropriate disjoint structural classes and closed counts in those higher layers.

The supplementary material records examples, classifier material, the finite check, optional aggregate reformulations, and comparisons with prescribed-shift and aggregate word-level formulas.

\section*{Declarations}
\textbf{Funding.} This research did not receive any specific grant from funding agencies in the public, commercial, or not-for-profit sectors.

\textbf{Declaration of competing interests.} The author declares that there are no known competing financial interests or personal relationships that could have appeared to influence the work reported in this paper.

\textbf{Declaration of generative AI and AI-assisted technologies in the manuscript preparation process.} During the preparation of this work, the author used ChatGPT for editorial assistance and formatting support. The author reviewed and edited the content as needed and takes full responsibility for the published article.

\textbf{Data and code availability.} No external input datasets or separate code/data artifacts are supplied.  The supplementary material documents a finite check of the stated enumerative formulas by giving the deterministic enumeration scan, pseudocode, aggregate count table, representative sample records, field specification, and aggregate output summary.  These details document the small-length check and are not a proof dependency for the main classification and enumeration theorem.

\end{document}


\let\WriteBookmarks\relax
\let\printorcid\relax
\ExplSyntaxOn
\cs_if_exist:NT \g_stm_nologo_bool
  { \bool_gset_true:N \g_stm_nologo_bool }
\ExplSyntaxOff

\ExplSyntaxOn
\cs_set:Npn \__first_footerline:
  {
    \group_begin:
    \small
    \sffamily
    \ifnum\theblind>0\relax
    \else
      \__short_authors:
    \fi
    \group_end:
  }
\ExplSyntaxOff

\renewcommand{\floatpagefraction}{1}
\renewcommand{\textfraction}{.001}

\shorttitle{Supplement: Minimum nontrivial rotational Hamming distance $2$}
\shortauthors{Makhynko}

\title[mode=title]{Supplementary material for binary necklaces with minimum nontrivial rotational Hamming distance $2$}

\author[1]{Mykola Makhynko}
\cormark[1]
\ead{mykola.makhynko@gmail.com}
\affiliation[1]{organization={Independent Researcher},
            city={Unionville},
            state={Ontario},
            country={Canada}}
\cortext[1]{Corresponding author}

\begin{abstract}
This supplementary material accompanies the main manuscript on binary necklaces with minimum nontrivial rotational Hamming distance $2$.  It provides equivalent aggregate formulations, worked examples, an explicit classifier, and a finite check for $2\le n\le 20$ documented here.  The finite-check section records the scan specification, representative data, count table, and aggregate summary.  The supplement also compares prescribed-shift counts with aggregate word-level formulas in the literature.
\end{abstract}

\begin{keywords}
binary necklaces \sep rotational Hamming distance \sep supplementary material \sep finite check \sep enumeration
\end{keywords}
\maketitle

\section*{Supplement contents}
The notation, result names, and class names follow the main manuscript.  This supplement contains optional consequences, reformulations, the finite check, and comparisons for the distance-$2$ theorems proved in the main manuscript.  It is not a proof dependency for the main classification and enumeration theorem.

The supplement is organized as follows.  Section~S1 states the scope of the supplement.  Sections~S2 and~S3 give uniform aggregate and Lambert-series reformulations.  Section~S4 gives the practical classifier and worked examples, including the quotient--background construction.  Section~S5.1 documents the finite check for $2\le n\le 20$, and Sections~S5.2 and~S5.3 give the prescribed-shift and aggregate comparisons with Gabri\'c's formulas.


\section{Scope of this supplement}
\label{sec:supp-scope}
This supplement accompanies the distance-$2$ enumeration theorem and its comparisons.  It first gives equivalent aggregate formulations, including a compact uniform rewrite of the expanded aggregate count.  It then provides worked examples, a practical classifier, and algebraic comparisons with Gabri\'c's aggregate word-level formulas.  The finite check for $2\le n\le 20$ is documented in \Cref{sec:smalln}.

\section{Uniform aggregate formulation}
\begin{corollary}[Uniform divisor-sum formulation of the aggregate corollary]
\label{cor:countD2-uniform}
Let
\[
\PrimCt{\ell}:=\sum_{d\mid \ell}\mu(d)\,2^{\ell/d},
\]
denote the number of primitive binary words of length~$\ell$ (cf. the primitive-word formula in the main manuscript).
Then for every $n\ge 2$,
\[
|\Dtwo(n)|
=\One_{2\mid n}\,2^{\frac n2-1}
+\sum_{m\mid n}\frac{\varphi(m)}{4}\,2^{\frac{n}{m}-1}\,\bigl(m-3+\lvert m-3\rvert\bigr)
+\sum_{\tau\mid n}
  \left(\sum_{d\mid \tau}\mu(d)\,2^{\tau/d}\right)\One_{2\nmid (n/\tau)}
-\sum_{d\mid n}\mu(d)\,2^{n/d}.
\]
Equivalently, the compact primitive-count form is
\[
|\Dtwo(n)|
=\One_{2\mid n}\,2^{\frac n2-1}
+\sum_{m\mid n}\frac{\varphi(m)}{4}\,2^{\frac{n}{m}-1}\,\bigl(m-3+\lvert m-3\rvert\bigr)
+\sum_{\tau\mid n} \PrimCt{\tau}\,\One_{2\nmid (n/\tau)}-\PrimCt{n}.
\]
\end{corollary}

\begin{proof}
Starting from the aggregate count corollary, rewrite the restricted sum over $m\ge 4$ using
\[
(m-3)\,\One_{m\ge 4}=\max(m-3,0)=\frac{m-3+\lvert m-3\rvert}{2}.
\]
For the M\"obius-sum terms, observe that
\[
\sum_{\tau\mid n} \PrimCt{\tau}\,\One_{2\nmid (n/\tau)}-\PrimCt{n}
=
\One_{3\mid n}\PrimCt{n/3}+\sum_{\substack{\tau\mid n\\ (n/\tau)\ \mathrm{odd}\\ n/\tau\ge 5}}\PrimCt{\tau},
\]
since the odd quotients $n/\tau$ are exactly $1,3,5,\dots$ and the subtraction removes the $\tau=n$ (quotient $1$) contribution.
Expanding $\PrimCt{\tau}=\sum_{d\mid \tau}\mu(d)\,2^{\tau/d}$ and $\PrimCt{n}=\sum_{d\mid n}\mu(d)\,2^{n/d}$ in this compact form gives the displayed M\"obius--totient version.
\end{proof}

\begin{remark}[Uniform divisor-sum form]
\Cref{cor:countD2-uniform} is the same aggregate formula rewritten as a divisor sum over all quotients.  The first display is fully expanded in $\mu$ and $\varphi$.  The second display records the compact primitive-count notation.  The totient term counts the SP pieces arising from the genuine-interval branch.  The primitive-word term combines the odd-quotient defect--background slices; it contributes the order-$3$ SP boundary case, the SUB slices for odd quotients at least $5$, and the subtraction of $\PrimCt{n}$ that removes the trivial quotient $M=1$.
\end{remark}

\section{Lambert series and asymptotic consequences}
The classwise formulas also yield a Lambert-series identity and the coarse asymptotic subsequences needed below. We record them here as brief consequences of the aggregate count corollary, after the classification and exact counts are in place.

\paragraph{Lambert-series generating function for \texorpdfstring{$|\Dtwo(n)|$}{abs D_2(n)}}
\label{sec:lambert}

The next consequence is a Lambert-series identity expressing the total necklace counts in a single series with the explicit coefficient sequence used in the asymptotic remark below.

For an arithmetic function $f:\ZZ_{>0}\to\ZZ$ write
\[
\mathcal{L}[f](x):=\sum_{n\ge 1} f(n)\,\frac{x^n}{1-x^n},
\qquad
\mathcal{L}_2[f](x):=\sum_{n\ge 1} f(n)\,\frac{x^n}{1-2x^n},
\]
for the associated Lambert series. Let
\begin{equation}
\label{eq:An-def}
A(n):=\sum_{d\mid n} d\,\varphi(d).
\end{equation}
Below we apply these operators only to the two specific coefficient sequences
\[
n\mapsto |\Dtwo(n)|,
\qquad
n\mapsto \frac12 A(n)-\frac32 n+3-\frac12\One_{2\mid n},
\]
and for both of these sequences the displayed series converge absolutely whenever $|x|<\tfrac12$.

\begin{theorem}[Single Lambert-series identity for $|\Dtwo(n)|$]
\label{thm:lambert-D2}
With the convention $|\Dtwo(1)|=0$ from the definition of $\Dtwo(1)$, for $|x|<\tfrac12$ the following identity of absolutely convergent Lambert series holds:
\begin{equation}
\label{eq:lambert-D2}
\sum_{n\ge 1} |\Dtwo(n)|\,\frac{x^n}{1-x^n}
=
\sum_{n\ge 2}
\left(
\frac12 A(n)-\frac32 n+3-\frac12\One_{2\mid n}
\right)\frac{x^n}{1-2x^n}.
\end{equation}
In particular, the Lambert series of $|\Dtwo(n)|$ is expressed as a \emph{single} Lambert series with an explicit arithmetic coefficient.
\end{theorem}

\begin{proof}
Both Lambert series in \eqref{eq:lambert-D2} converge absolutely for $|x|<\tfrac12$, so it suffices to compare coefficients of $x^N$.  To keep the divisor names separate, we use the following convention:
\[
\begin{array}{ll}
r & \text{argument of the aggregate count }|\Dtwo(r)|,\\
q & \text{half-turn quotient in the even case},\\
m & \text{quotient/order parameter from the classwise formulas},\\
s & \text{Lambert denominator divisor in a term }2^{N/s-1}.
\end{array}.
\]
With this convention the left-hand side contributes
\[
F(N):=\sum_{r\mid N}|\Dtwo(r)|.
\]
The following paragraphs compute the three pieces of this coefficient.

For $N\ge 2$, substitute the aggregate count corollary.  Its subgroup-shaped terms are already necklace-level contributions, so no further word-to-necklace quotient is taken.  The coefficient extraction has three pieces.  The HT term comes from the even divisors $r$ of $N$.  The odd piece merges the order-$3$ SP boundary case and the odd SUB terms, both indexed by the odd quotient $m$.  The SP contribution from the genuine-interval branch is obtained by extending to all final Lambert divisors $s$ and subtracting the missing cases $m=1,2,3$.  Thus,
\[
F(N)=F_{\mathrm{HT}}(N)+F_{\mathrm{odd}}(N)+F_{\mathrm{SP}}(N),
\]
where
\[
F_{\mathrm{HT}}(N)=\sum_{\substack{s\mid N\\ 2\mid s}}2^{N/s-1},
\qquad
F_{\mathrm{odd}}(N)=
\sum_{\substack{m\mid N\\ m=3}}\ \sum_{\tau\mid (N/m)}\PrimCt{\tau}
+\sum_{\substack{m\mid N\\ m\ \mathrm{odd},\ m\ge 5}}\ \sum_{\tau\mid (N/m)}\PrimCt{\tau},
\]
and
\[
F_{\mathrm{SP}}(N)=\sum_{s\mid N}\left(\sum_{\substack{m\mid s\\ m\ge 4}}\frac{\varphi(m)}{2}(m-3)\right)2^{N/s-1}.
\]

We first isolate the half-turn coefficient piece.
\begin{claim}
\label{clm:lambert-D2-HT}
For every $N\ge 2$,
\[
F_{\mathrm{HT}}(N)=\sum_{\substack{s\mid N\\ 2\mid s}}2^{N/s-1}.
\]
\end{claim}
\begin{proof}
The half-turn term from the aggregate count corollary, evaluated in the outer sum over $r\mid N$, contributes only when $r$ is even.  Writing $r=2q$, summing over all $r\mid N$ is equivalent to summing over $q\mid N/2$, and the contribution is
\[
\One_{2\mid N}\sum_{q\mid N/2}2^{q-1}.
\]
Now set $s:=N/q$.  Since $q\mid N/2$, the divisor $s$ of $N$ is even, and this substitution is a bijection between $q\mid N/2$ and even $s\mid N$.  Hence
\[
F_{\mathrm{HT}}(N)
=\One_{2\mid N}\sum_{q\mid N/2}2^{q-1}
=\sum_{\substack{s\mid N\\ 2\mid s}}2^{N/s-1},
\]
which proves the claim.
\end{proof}

We next isolate the merged odd-quotient primitive contribution.
\begin{claim}
\label{clm:lambert-D2-odd}
For every $N\ge 2$,
\[
F_{\mathrm{odd}}(N)=\sum_{\substack{s\mid N\\ s\ \mathrm{odd},\ s\ge 3}}2^{N/s}
=\sum_{\substack{s\mid N\\ s\ge 2}}\bigl(2-2\One_{2\mid s}\bigr)2^{N/s-1}.
\]
\end{claim}
\begin{proof}
Set
\[
F_{\mathrm{SP},3}(N):=\One_{3\mid N}\sum_{\tau\mid (N/3)}\PrimCt{\tau},
\qquad
F_{\mathrm{SUB}}(N):=\sum_{\substack{m\mid N\\ m\ \mathrm{odd},\ m\ge 5}}\ \sum_{\tau\mid (N/m)}\PrimCt{\tau}.
\]
Then
\[
F_{\mathrm{odd}}(N)
=F_{\mathrm{SP},3}(N)+F_{\mathrm{SUB}}(N)
=\sum_{\substack{m\mid N\\ m=3}}\ \sum_{\tau\mid (N/m)}\PrimCt{\tau}
+\sum_{\substack{m\mid N\\ m\ \mathrm{odd},\ m\ge 5}}\ \sum_{\tau\mid (N/m)}\PrimCt{\tau}
=\sum_{\substack{m\mid N\\ m\ \mathrm{odd},\ m\ge 3}}\ \sum_{\tau\mid (N/m)}\PrimCt{\tau},
\]
because the only odd quotient missing from the SUB range is $m=3$, namely the order-$3$ SP boundary case, while the quotient-$1$ term remains excluded and never contributes to positive distance. By the primitive-word formula in the main manuscript, for each odd divisor $m$ of $N$ with $m\ge 3$,
\[
\sum_{\tau\mid (N/m)}\PrimCt{\tau}=2^{N/m}.
\]
Relabel this odd quotient divisor as the final Lambert divisor $s$.  Then
\[
\begin{aligned}
F_{\mathrm{odd}}(N)
&=\sum_{\substack{s\mid N\\ s\ \mathrm{odd},\ s\ge 3}}2^{N/s} 
&=\sum_{\substack{s\mid N\\ s\ \mathrm{odd},\ s\ge 3}}2\cdot 2^{N/s-1} 
&=\sum_{\substack{s\mid N\\ s\ge 2}}\bigl(2-2\One_{2\mid s}\bigr)2^{N/s-1},
\end{aligned},
\]
where the last expression records the coefficient $2$ on odd divisors $s\ge 3$ and the coefficient $0$ on even divisors. This proves the claim.
\end{proof}

We now isolate the generic SP coefficient piece.
\begin{claim}
\label{clm:lambert-D2-SP}
For every $N\ge 2$,
\[
F_{\mathrm{SP}}(N)=\sum_{s\mid N}\frac12\bigl(A(s)-3s+2+\One_{2\mid s}\bigr)2^{N/s-1}.
\]
\end{claim}
\begin{proof}
For the generic SP term, the divisor restriction $m\ge 4$ reflects that the order-$3$ SP boundary case was already absorbed into $F_{\mathrm{odd}}(N)$, so only the genuine-interval SP divisors remain here. Before collecting by final Lambert divisor, this contribution is
\[
\sum_{r\mid N}\sum_{\substack{m\mid r\\ m\ge 4}}\frac{\varphi(m)}{2}(m-3)2^{r/m-1}.
\]
Write $q=r/m$ and set $s=N/q$. Then $m\mid r\mid N$ is equivalent to $m\mid s\mid N$, and $2^{r/m-1}=2^{N/s-1}$. Therefore, for each final Lambert divisor $s\mid N$, set
\[
C_{\mathrm{SP}}(s):=\sum_{\substack{m\mid s\\ m\ge 4}}\frac{\varphi(m)}{2}(m-3)
=
\frac12\sum_{\substack{m\mid s\\ m\ge 4}}\bigl(m\varphi(m)-3\varphi(m)\bigr).
\]
Then
\[
F_{\mathrm{SP}}(N)=\sum_{s\mid N}C_{\mathrm{SP}}(s)\,2^{N/s-1}.
\]
To rewrite $C_{\mathrm{SP}}(s)$, extend both divisor sums to all divisors of $s$ and record the omitted $m=1,2,3$ contributions explicitly:
\[
\sum_{\substack{m\mid s\\ m\in\{1,2,3\}}}m\varphi(m)
=1+2\One_{2\mid s}+6\One_{3\mid s},
\qquad
\sum_{\substack{m\mid s\\ m\in\{1,2,3\}}}\varphi(m)
=1+\One_{2\mid s}+2\One_{3\mid s}.
\]
Hence
\[
2C_{\mathrm{SP}}(s)
=
\left(\sum_{m\mid s}m\varphi(m)-\sum_{\substack{m\mid s\\ m\in\{1,2,3\}}}m\varphi(m)\right)
-3\left(\sum_{m\mid s}\varphi(m)-\sum_{\substack{m\mid s\\ m\in\{1,2,3\}}}\varphi(m)\right)
=
\bigl(A(s)-1-2\One_{2\mid s}-6\One_{3\mid s}\bigr)
-3\bigl(s-1-\One_{2\mid s}-2\One_{3\mid s}\bigr).
\]
Here we used $\sum_{m\mid s}\varphi(m)=s$ and the definition of $A(s)=\sum_{m\mid s}m\varphi(m)$. Therefore
\[
\begin{aligned}
2C_{\mathrm{SP}}(s)
&=A(s)-1-2\One_{2\mid s}-6\One_{3\mid s}
-3s+3+3\One_{2\mid s}+6\One_{3\mid s}
&=A(s)-3s+2+\One_{2\mid s},
\end{aligned},
\]
so
\[
C_{\mathrm{SP}}(s)=\frac12\bigl(A(s)-3s+2+\One_{2\mid s}\bigr).
\]
Substituting this coefficient identity into the divisor sum for $F_{\mathrm{SP}}(N)$ gives
\[
F_{\mathrm{SP}}(N)=\sum_{s\mid N}\frac12\bigl(A(s)-3s+2+\One_{2\mid s}\bigr)2^{N/s-1},
\]
which proves the claim.
\end{proof}

Finally, \Cref{clm:lambert-D2-HT,clm:lambert-D2-odd,clm:lambert-D2-SP} identify the three divisorwise contributions to
\[
F(N)=F_{\mathrm{HT}}(N)+F_{\mathrm{odd}}(N)+F_{\mathrm{SP}}(N).
\]
Therefore, for each final Lambert divisor $s\mid N$ with $s\ge 2$, the total coefficient of $2^{N/s-1}$ is
\[
C(s):=
\One_{2\mid s}
+\bigl(2-2\One_{2\mid s}\bigr)
+\frac12\bigl(A(s)-3s+2+\One_{2\mid s}\bigr).
\]
This simplifies to
\[
C(s)=\frac12A(s)-\frac32s+3-\frac12\One_{2\mid s}.
\]
Hence, for every $N\ge 2$,
\[
F(N)=\sum_{\substack{s\mid N\\ s\ge 2}}
\left(\frac12A(s)-\frac32s+3-\frac12\One_{2\mid s}\right)2^{N/s-1}.
\]
This is exactly the coefficient of $x^N$ on the right-hand side of \eqref{eq:lambert-D2}, since
\[
\frac{x^s}{1-2x^s}=\sum_{j\ge 1}2^{j-1}x^{js},
\]
contributes $2^{N/s-1}$ whenever $s\mid N$. For $N=1$, both sides have coefficient $0$ because $|\Dtwo(1)|=0$ and the right-hand sum starts at $s=2$. Therefore, all coefficients agree.
\end{proof}

\begin{lemma}[Even-case dominant term]
\label{lem:countD2-even-bound}
Let $d(n)$ denote the number of positive divisors of $n$. For even $n\ge 2$,
\[
|\Dtwo(n)|=2^{n/2-1}+O\!\left(\One_{3\mid n}\,2^{n/3}+n^{2}d(n)\,2^{n/4}\right).
\]
\end{lemma}

\begin{proof}
Assume that $n$ is even. The aggregate count corollary gives the HT contribution exactly as $2^{n/2-1}$.
For the order-$3$ SP boundary case,
\[
\One_{3\mid n}\sum_{d\mid (n/3)} \mu(d)\,2^{\frac{n}{3d}}
=\One_{3\mid n}\PrimCt{n/3}
=O\!\left(\One_{3\mid n}\,2^{n/3}\right),
\]
because $\PrimCt{\ell}$ counts primitive binary words of length $\ell$ and is therefore at most $2^\ell$.

For the generic interval SP contribution, every divisor $m\mid n$ with $m\ge 4$ satisfies $n/m\le n/4$, so
\[
\frac{\varphi(m)}{2}\,2^{\frac{n}{m}-1}(m-3)\le \frac{m^2}{2}\,2^{n/4-1}\le n^2 2^{n/4}.
\]
There are at most $d(n)$ such divisors. Hence
\[
\sum_{\substack{m\mid n\\ m\ge 4}}\frac{\varphi(m)}{2}\,2^{\frac{n}{m}-1}(m-3)
=O\!\left(n^2 d(n)\,2^{n/4}\right).
\]

For the remaining SUB contribution, every admissible divisor $\tau\mid n$ has odd quotient $n/\tau\ge 5$, so $\tau\le n/5$. Again using $\PrimCt{\tau}\le 2^\tau$,
\[
\sum_{\substack{\tau\mid n\\ (n/\tau)\ \mathrm{odd}\\ n/\tau\ge 5}}\sum_{d\mid \tau}\mu(d)\,2^{\tau/d}
=\sum_{\substack{\tau\mid n\\ (n/\tau)\ \mathrm{odd}\\ n/\tau\ge 5}}\PrimCt{\tau}
=O\!\left(d(n)\,2^{n/5}\right),
\]
which is absorbed by $O\!\left(n^2 d(n)\,2^{n/4}\right)$.
Combining the four contributions proves the claim.
\end{proof}

\begin{lemma}[Odd multiples of $3$: dominant primitive term]
\label{lem:countD2-odd3-bound}
Let $d(n)$ denote the number of positive divisors of $n$. For odd $n$ with $3\mid n$,
\[
|\Dtwo(n)|=\PrimCt{n/3}+O\!\left(n^{2}d(n)\,2^{n/5}\right).
\]
Consequently,
\[
|\Dtwo(n)|=2^{n/3}+O\!\left(d(n)\,2^{n/6}+n^{2}d(n)\,2^{n/5}\right).
\]
\end{lemma}

\begin{proof}
Assume that $n$ is odd and divisible by $3$. In the aggregate count corollary, there is no HT contribution, and the order-$3$ SP boundary case contributes exactly $\PrimCt{n/3}$.

For the generic interval SP contribution, every divisor $m\mid n$ with $m\ge 4$ is odd and therefore satisfies $m\ge 5$, so $n/m\le n/5$. Hence
\[
\frac{\varphi(m)}{2}\,2^{\frac{n}{m}-1}(m-3)\le \frac{m^2}{2}\,2^{n/5-1}\le n^2 2^{n/5}.
\]
There are at most $d(n)$ such divisors. Therefore
\[
\sum_{\substack{m\mid n\\ m\ge 4}}\frac{\varphi(m)}{2}\,2^{\frac{n}{m}-1}(m-3)
=O\!\left(n^{2}d(n)\,2^{n/5}\right).
\]

For the SUB contribution, every admissible divisor $\tau\mid n$ has odd quotient $n/\tau\ge 5$, hence $\tau\le n/5$. Using $\PrimCt{\tau}\le 2^\tau$ again,
\[
\sum_{\substack{\tau\mid n\\ (n/\tau)\ \mathrm{odd}\\ n/\tau\ge 5}}\sum_{d\mid \tau}\mu(d)\,2^{\tau/d}
=\sum_{\substack{\tau\mid n\\ (n/\tau)\ \mathrm{odd}\\ n/\tau\ge 5}}\PrimCt{\tau}
=O\!\left(d(n)\,2^{n/5}\right),
\]
which is absorbed by the previous error term. This proves the first estimate.

For the reformulation, apply the primitive-word formula in the main manuscript with $\ell=n/3$:
\[
\PrimCt{n/3}
=2^{n/3}+\sum_{\substack{e\mid (n/3)\\ e\ge 2}}\mu(e)\,2^{\frac{n}{3e}}
=2^{n/3}+O\!\left(d(n)\,2^{n/6}\right),
\]
since every divisor $e\ge 2$ gives $(n/3)/e\le n/6$. Combining this with the first estimate gives the second formula.

When the odd-prime special case is $p=3$, the generic SP and SUB sums are empty, so the aggregate count corollary gives $|\Dtwo(3)|=\PrimCt{1}=2$, in agreement with the odd-prime specialization. For odd primes $p\ge 5$, this odd-$3$-divisible regime does not apply.
\end{proof}

\begin{remark}[Arithmetic and asymptotic corollaries]
The function $A(n)=\sum_{d\mid n} d\varphi(d)$ is multiplicative by standard divisor-sum arithmetic; this fact is independent of the aggregate count corollary. For a prime power $p^k$,
\[
A(p^k)=1+\sum_{j=1}^k p^j\varphi(p^j)=\frac{p^{2k+1}+1}{p+1},
\qquad
A(n)=\prod_{p^k\parallel n}\frac{p^{2k+1}+1}{p+1}.
\]
Using the convention $|\Dtwo(1)|=0$, write $d(n)$ for the number of positive divisors of $n$. By \Cref{lem:countD2-even-bound,lem:countD2-odd3-bound},
\[
|\Dtwo(n)|=2^{n/2-1}+O\!\left(\One_{3\mid n}\,2^{n/3}+n^{2}d(n)\,2^{n/4}\right)
\qquad (n\ \text{even}),
\]
while for odd $n$ with $3\mid n$,
\[
|\Dtwo(n)|=\PrimCt{n/3}+O\!\left(n^{2}d(n)\,2^{n/5}\right)=2^{n/3}+O\!\left(d(n)\,2^{n/6}+n^{2}d(n)\,2^{n/5}\right),
\]
which are the only asymptotic subsequences needed here. The separate odd-prime exact formula remains isolated in the odd-prime specialization. Since the classical necklace-count formula gives $|\Neck{n}|=2^n/n+O(2^{n/2})$, the class $\Dtwo(n)$ is exponentially sparse among all binary necklaces, with
\[
\frac{|\Dtwo(n)|}{|\Neck{n}|}\sim \frac{n}{2}\,2^{-n/2}\qquad (n\ \text{even}),
\qquad
\frac{|\Dtwo(n)|}{|\Neck{n}|}=O\!\left(n\,2^{-2n/3}\right)\qquad (n\ \text{odd},\ 3\mid n).
\]
\end{remark}

\section{Practical classifier and worked examples}
\label{sec:worked-examples}

\begin{corollary}[Practical distance-$2$ classifier]
\label{cor:classifier}
Let $w\in\B^n$ be a representative of a binary necklace of length $n$ with $\minDist(w)=2$.
Compute $\Ktwo(w)$ by scanning all nonzero shifts and counting mismatches between $w$ and $\rot^k w$, stopping the scan for a given $k$ once more than two mismatches are found.
After the trichotomy theorem in the main manuscript, the theorem-level classifier has the following three terminal classes:
\begin{itemize}[leftmargin=2em]
\item If $n$ is even and $n/2\in\Ktwo(w)$, classify the necklace as HT and set $\tau=\tautwo(w)=\gcdop(n,\Ktwo(w))$;
\item If the HT condition is not satisfied and $|\Ktwo(w)|=2$, write $\Ktwo(w)=\{k,n-k\}$ with $1\le k<n/2$ and classify the necklace as SP together with this $k$;
\item If both preceding conditions fail, classify the necklace as SUB with strict period $\tau=\tautwo(w)=\gcdop(n,\Ktwo(w))$.
\end{itemize}
This is the post-trichotomy case split: it uses the theorem's conclusion that the residual MP class is empty at distance $2$.  A direct implementation for examples runs in $O(n^2)$ time in the worst case and $O(n)$ space; early stopping can reduce the number of comparisons.

The independent finite check in \Cref{sec:smalln} does not use this post-trichotomy classifier to rule out MP.  Instead, it computes the full distance profile, forms $K_{\mathrm{scan}}(w)$, and then tests all four priority branches directly: HT first, SP second, SUB only after checking $K_{\mathrm{scan}}(w)=\tau\Zn\setminus\{0\}$, and MP otherwise.  The comparison with the three-class theorem-level classifier is made only after those four-way labels and counters have been produced.
\end{corollary}

\begin{example}[Worked quotient--background construction]
\label{ex:defect-background-worked}
Take $n=12$ and $\tau=3$, so we choose the minimizing shift $k=\tau=3$ and the quotient order is
\[
M=n/\tau=4.
\]
Choose the primitive quotient--background word
\[
c=(0,0,1)\in\B^3.
\]
Here $c_0=0$.  By the primitive-readout corollary in the main manuscript, part~(b), the $1$-set on the active $\tau$-orbit is a singleton in the normalized representative.
The period subgroup of the quotient--background word $c$ is
\[
P=\{e\in\Zquot{3}:\rot^e c=c\}=\{0\}=3\Zquot{3},
\]
so part~(a) of the primitive-readout corollary gives $\rho=\tau=3$.
The orbit-constant lift is
\[
v=001001001001,
\]
and the normalized representative $w=w(c)$ specified by the quotient--background construction in the main manuscript is obtained by flipping the unique defect position on the active $\tau$-orbit $\{0,3,6,9\}$:
\[
w=101001001001.
\]
Equivalently,
\[
(w_0,w_3,w_6,w_9)=(1,0,0,0),
\]
while the other two $\tau$-orbits carry the background bits $c_1=0$ and $c_2=1$ constantly.
Thus the defect is recorded exactly at index $0$, and the quotient--background construction recovers the minimizer set from the background period subgroup as
\[
\Ktwo(w)=\rho\Zn\setminus\{0\}=3\Zn\setminus\{0\}=\{3,6,9\}.
\]
Since $c$ is primitive and $M=4$ is even, the defect-counting corollary in the main manuscript, part~(b), places the resulting necklace in the refined half-turn slice $\DtwoHT(12;3)$ and shows that this primitive quotient--background word contributes one necklace at that level.
\end{example}

\Cref{tab:d2-worked-types} illustrates the classifier from \Cref{cor:classifier}.  The table records sample representatives, not necessarily the lexicographically least rotations, for the main classifier outputs and the order-$3$ SP boundary case.  The table also displays the length and the interval or quotient-order datum that makes the classifier decision immediate.

\begin{table}[h]
\centering
\small
\begin{tabular}{cccccc}
\hline
Example type & Length $n$ & Representative $w$ & $\Ktwo(w)$ & Interval / quotient datum & Classifier class \\
\hline
HT & $8$ & \texttt{00010011} & $\{4\}$ & half-turn $4=n/2$ & HT \\
SP & $5$ & \texttt{11000} & $\{1,4\}$ & prime interval $w_{1,2}$ & SP \\
order-$3$ SP boundary case & $3$ & \texttt{100} & $\{1,2\}=\Zquot{3}\setminus\{0\}$ & quotient order $M=3$ & SP \\
SUB & $5$ & \texttt{10000} & $\{1,2,3,4\}$ & $\tau=1$, $M=5$ & SUB \\
\hline
\end{tabular}
\caption{Short sample representatives for HT, SP, the order-$3$ SP boundary case, and SUB.  These words are not necessarily the lexicographically least rotations.  The first column records the example type, and the final column records the classifier class.}
\label{tab:d2-worked-types}
\end{table}

A full nontrivial classifier trace is provided by the subgroup-shaped representative
\[
w=101001001001,
\]
from the worked quotient--background construction above, obtained there from the primitive quotient--background word $c=(0,0,1)$ at $(n,\tau,M)=(12,3,4)$. A direct scan of the nonzero shifts gives
\[
d_w(3)=d_w(6)=d_w(9)=2,
\qquad
d_w(k)=8\ \text{for}\ k\in\{1,2,4,5,7,8,10,11\},
\]
so
\[
\Ktwo(w)=\{3,6,9\}=3\Zn\setminus\{0\}.
\]
The practical classifier from \Cref{cor:classifier} now runs in four steps:
\begin{enumerate}[label=(\arabic*),leftmargin=2.5em]
\item the half-turn $6=n/2$ belongs to $\Ktwo(w)$, so the first classifier step already classifies the necklace as HT and rules out SP;
\item because $|\Ktwo(w)|=3>1$, this is not the pure boundary $\{n/2\}$ but a subgroup-shaped HT example;
\item the subgroup divisor is
\[
\tau=\tautwo(w)=\gcdop(12,\Ktwo(w))=3,
\qquad
M=12/\tau=4;
\]
\item since $M$ is even and at least $4$, the primitive-readout and defect-counting corollaries place the necklace in the mixed half-turn slice $\DtwoHT(12;3)$.
\end{enumerate}
Thus this worked example shows an HT classification arising from the subgroup-shaped minimizer set $3\Zn\setminus\{0\}=\{3,6,9\}$ rather than from the pure boundary $\{n/2\}$. By contrast, the prime-cycle representatives in \Cref{tab:d2-worked-types} come from the prime-interval normal form in the main manuscript: \texttt{11000} is $w_{1,2}$ and so lies in SP, while \texttt{10000} is $w_{1,1}$ and has $\Ktwo(w)=\Zquot{5}\setminus\{0\}$, so the quotient-order test gives SUB.

\section{Finite check and algebraic comparisons}
\label{sec:impl}

This section contains the finite check for $2\le n\le 20$, the prescribed-shift comparison with Gabri\'c's word-level count, and, after primitive-orbit division, the aggregate and pure-half-turn necklace-level comparisons using Gabri\'c's notation.

For the finite check, an exhaustive orbit scan enumerates binary words of each length $2\le n\le 20$, retains the lexicographically least rotation from each orbit, and computes the full distance profile over nonzero shifts before assigning any class label.  \Cref{sec:smalln} gives the finite-check documentation described in the supplement contents.

The prescribed-shift and aggregate-comparison subsections give exact comparisons derived from lemmas and count formulas in the main manuscript, using the same descriptive result names as the main text.

\subsection{Finite check for small \texorpdfstring{$n$}{n}}
\label{sec:smalln}

For $2\le n\le 20$, \Cref{tab:smalln-exhaustive-counts} lists the HT, SP, SUB, MP, and total counts obtained from an exhaustive orbit scan.  The rows compare these counts with the class-count theorems, the trichotomy theorem, and the aggregate count corollary.

The scan is deterministic: it enumerates the lexicographically least representative of each binary necklace, computes the full distance profile over nonzero shifts, and then applies the HT/SP/SUB/MP priority order.  It uses the distance formula \texttt{DIST}, the order $0<1$, and no external input data.  The class formulas are used only for comparison, not as inputs to the scan.  The recorded per-necklace fields are listed below and illustrated by the sample records and table.
\begin{enumerate}[label=(\arabic*),leftmargin=2.5em]
\item For each $n\in\{2,\dots,20\}$, enumerate all binary words $u\in\B^n$, form the rotation orbit of $u$, and retain $u$ only when it is the lexicographically least rotation in that orbit. This keeps exactly one representative of each necklace.
\item For each retained representative $w$, compute the full distance profile over nonzero shifts
\[
\bigl(d_w(1),d_w(2),\dots,d_w(n-1)\bigr),
\]
exactly from the rotational-distance definition in the main manuscript; neither class formulas nor early stopping are used in this finite check.
\item Let
\[
\delta_{\mathrm{scan}}(w):=\min_{1\le k<n}d_w(k),
\qquad
K_{\mathrm{scan}}(w):=\{k\in\{1,\dots,n-1\}:d_w(k)=\delta_{\mathrm{scan}}(w)\}.
\]
Keep $w$ in the table exactly when $\delta_{\mathrm{scan}}(w)=2$. For those retained representatives, $K_{\mathrm{scan}}(w)=\Ktwo(w)$ by definition.
\item Assign the retained representative to a class directly from the distance-$2$ class taxonomy in the main manuscript: first HT if $n$ is even and $n/2\in K_{\mathrm{scan}}(w)$; otherwise SP if $K_{\mathrm{scan}}(w)=\{k,n-k\}$ for some $k$; otherwise set $\tau=\gcdop(n,K_{\mathrm{scan}}(w))$ and test whether
\[
K_{\mathrm{scan}}(w)=\tau\Zn\setminus\{0\}.
\]
If this equality holds, record SUB; otherwise record MP.
\item For every retained distance-$2$ representative, record the following per-necklace fields:
\begin{itemize}[leftmargin=2em]
\item profile data: $n$, the full profile, $\delta_{\mathrm{scan}}$, $K_{\mathrm{scan}}(w)=\Ktwo(w)$, the Hamming weight, and the adjacent run count $d_w(1)$;
\item class data: the class label, $\tau=\gcdop(n,K_{\mathrm{scan}}(w))$, the quotient $M=n/\tau$, and the minimizer orders $n/\gcdop(n,k)$;
\item quotient data: the normalized minimizer set $\{k/\tau:k\in K_{\mathrm{scan}}(w)\}$ and its unit-action signature.
\end{itemize}
\item Increment the corresponding HT/SP/SUB/MP counter for that $n$. The table, sample records, and the aggregate output summary are outputs of this scan; a nonzero MP counter would appear before comparison with the main-manuscript formulas.
\end{enumerate}
The pseudocode specifies the implementation details.  The generator uses the standard $1$-based recursion variables; distance computations use zero-based indices modulo $n$.  The procedure \texttt{GENERATE-NECKLACES} is the standard lexicographic necklace recursion and emits the lexicographically least binary representative of each rotation orbit.

All comparisons use the order $0<1$, and \texttt{gcd(n,K)} denotes the greatest common divisor of $n$ and the elements of $K$.  Thus the recursion, \texttt{DIST}, and the HT/SP/SUB/MP priority order determine the computation.  The table is the aggregate output; the sample record and field specification identify the stored per-necklace fields.
\begin{center}
\begin{minipage}{0.98\linewidth}
\small
\begin{verbatim}
procedure GENERATE-NECKLACES(n):
    a[0], a[1], ..., a[n] := 0
    procedure GEN(t,p):
        if t > n:
            if n mod p = 0:
                emit (a[1], a[2], ..., a[n])
        else:
            a[t] := a[t-p]
            GEN(t+1,p)
            for b from a[t-p]+1 to 1:
                a[t] := b
                GEN(t+1,t)
    GEN(1,1)

function DIST(w,k):
    n := length(w)
    return sum_{i=0}^{n-1} indicator(w[i] != w[(i-k) mod n])

function UNIT-SIGNATURE(Q,M):
    U := {a in {1,...,M-1}: gcd(a,M)=1}
    best := none
    for a in U:
        candidate := sorted set {(a*q) mod M: q in Q}
        if best is none or candidate is lexicographically smaller than best:
            best := candidate
    return best

procedure SMALL-DISTANCE-TWO-CHECK:
    for n from 2 to 20:
        count[HT], count[SP], count[SUB], count[MP] := 0,0,0,0
        records := empty list
        for w in GENERATE-NECKLACES(n):
            profile[k] := DIST(w,k) for k=1,...,n-1
            delta := min {profile[k]: k=1,...,n-1}
            if delta != 2: continue
            K := {k in {1,...,n-1}: profile[k] = 2}
            tau := gcd(n,K)
            if n is even and n/2 in K:
                label := HT
            else if there is k in {1,...,n-1} with K = {k,n-k}:
                label := SP
            else if K = {j in {1,...,n-1}: tau divides j}:
                label := SUB
            else:
                label := MP
            M := n/tau
            orders := sorted set {n/gcd(n,k): k in K}
            Q := sorted set {k/tau: k in K}
            signature := UNIT-SIGNATURE(Q,M)
            record (n,w,profile,delta,K,sum_i w[i],profile[1],
                    label,tau,M,orders,Q,signature)
            count[label] := count[label]+1
        output n, count[HT], count[SP], count[SUB], count[MP],
               count[HT]+count[SP]+count[SUB]+count[MP]
\end{verbatim}
\end{minipage}
\end{center}
The closed formulas from the main enumeration and class-count results, together with the post-trichotomy practical classifier in \Cref{cor:classifier}, are not used in the pseudocode; they enter only after the finite check produces the four-way HT/SP/SUB/MP counters and sample records, as theorem-level comparisons.

\paragraph{Finite-check documentation.}
The finite check is specified in this supplementary material by the deterministic scan, pseudocode, count table, sample records, field specification, and aggregate output summary.  No separate finite-check script, retained-record CSV, output transcript, or checksum file is supplied with the supplementary material.  This documentation records the small-length check and is not a proof dependency for the main classification and enumeration theorem.

The sample record and finite-check rows illustrate the per-necklace fields specified in the next paragraph.  The aggregate output summary records the checked range, procedure description, class totals, and cumulative total.  The cumulative total is the number of retained distance-$2$ necklace representatives found by the scan for $2\le n\le 20$.

\paragraph{Per-necklace record fields.}
For each retained distance-$2$ representative, the scan records entries in increasing $n$ and then by the lexicographically least rotation representative emitted by the necklace generator.  The recorded fields are
\begin{center}
\small
\begin{tabular}{l}
\texttt{n, representative, profile, delta, K, weight, d1, label, tau, M,}\\
\texttt{orders, Q, unit\_signature}
\end{tabular}
\end{center}
The list-valued fields \texttt{profile}, \texttt{K}, \texttt{orders}, \texttt{Q}, and \texttt{unit\_signature} are displayed as bracketed integer lists.  The sample record and finite-check rows illustrate representative entries for these fields.

\paragraph{Sample record.}
For the necklace represented by $w=101001001001$ in \Cref{ex:defect-background-worked}, the rotation-invariant fields used in the finite check have $n=12$, with the full distance profile over nonzero shifts summarized by
\[
d_w(3)=d_w(6)=d_w(9)=2,
\qquad
 d_w(k)=8\quad\text{for }k\in\{1,2,4,5,7,8,10,11\}.
\]
Thus $\delta_{\mathrm{scan}}=2$, $K_{\mathrm{scan}}(w)=\{3,6,9\}=\Ktwo(w)$, the Hamming weight is $5$, and $d_w(1)=8$.  The class fields are
\[
\tau=\gcdop(12,\{3,6,9\})=3,
\qquad M=12/3=4,
\qquad \{12/\gcdop(12,k):k\in K_{\mathrm{scan}}(w)\}=\{2,4\},
\]
and the quotient fields are $\{k/\tau:k\in K_{\mathrm{scan}}(w)\}=\{1,2,3\}\subseteq\Zquot{4}$, with unit-action signature $\{1,2,3\}$.

Since the half-turn $6=n/2$ lies in $K_{\mathrm{scan}}(w)$, the class label recorded by the finite-check procedure is HT.

The finite-check table below lists one lexicographically least representative for each classifier case or boundary type.  The minimizer column $K_{\mathrm{scan}}(w)$ records the nonzero shifts achieving the minimum distance, and the class fields $\tau$, $M$, and the priority rule determine the recorded label.  The full distance profile $(d_w(1),\dots,d_w(n-1))$ for these representatives is the raw scan data summarized in \Cref{tab:d2-worked-types} and the worked record above.  All rows come from the deterministic scan described above.
\begin{center}
\scriptsize
\begin{tabular}{llclccc p{3.4cm}}
\hline
Check type & $n$ & Representative & $K_{\mathrm{scan}}(w)$ & $\tau$ & $M$ & Label & Priority rule \\
\hline
pure HT & $8$ & \texttt{00010011} & $\{4\}$ & $4$ & $2$ & HT & $4=n/2\in K$ \\
mixed HT & $12$ & \texttt{001001001101} & $\{3,6,9\}$ & $3$ & $4$ & HT & $6=n/2\in K$ \\
generic SP & $5$ & \texttt{00011} & $\{1,4\}$ & $1$ & $5$ & SP & HT condition not satisfied; $K=\{1,5-1\}$ \\
order-$3$ SP boundary case & $3$ & \texttt{001} & $\{1,2\}=\Zquot{3}\setminus\{0\}$ & $1$ & $3$ & SP & HT condition not satisfied; $|K|=2$ has priority over SUB \\
SUB & $5$ & \texttt{00001} & $\{1,2,3,4\}=\Zquot{5}\setminus\{0\}$ & $1$ & $5$ & SUB & HT and SP conditions not satisfied; subgroup-shaped condition applies \\
\hline
\end{tabular}
\end{center}

The aggregate output below summarizes the deterministic scan.  The worked record and finite-check rows show how the recorded fields determine class assignments before any closed formula is consulted.  The profile, minimizer set, class label, quotient data, and signature shown here are invariant under rotation.

\subsubsection*{Finite-check output summary}

The following printed summary gives aggregate output from the deterministic scan described in this section.  It records the checked range, software environment, procedure description, class totals, and cumulative number of retained necklace representatives.  Class labels are computed from the full distance profile over nonzero shifts for each retained necklace representative, without using the closed formulas from the main manuscript.
\begin{center}
\begin{minipage}{0.98\linewidth}
\scriptsize
\begin{verbatim}
Checked range: 2 <= n <= 20
Software environment used for printed summary: Python 3.13.5 (CPython)
Procedure: deterministic scan described in this supplementary material; no closed formulas used
Record fields: n,representative,profile,delta,K,weight,d1,label,tau,M,orders,Q,unit_signature
n    HT   SP  SUB   MP total
 2    1    0    0    0    1
 3    0    2    0    0    2
 4    2    1    0    0    3
 5    0    4    2    0    6
 6    4    5    0    0    9
 7    0   12    2    0   14
 8    8   12    0    0   20
 9    0   24    2    0   26
10   16   22    2    0   40
11    0   40    2    0   42
12   32   40    0    0   72
13    0   60    2    0   62
14   64   57    2    0  123
15    0   94    8    0  102
16  128   80    0    0  208
17    0  112    2    0  114
18  256  147    2    0  405
19    0  144    2    0  146
20  512  144   12    0  668
cumulative 1023 1000 40 0 2063
\end{verbatim}
\end{minipage}
\end{center}

\begin{table}[h]
\centering
\small
\begin{tabular}{rrrrrr}
\hline
$n$ & HT & SP & SUB & MP & total \\
\hline
 2 & 1 & 0 & 0 & 0 & 1 \\
 3 & 0 & 2 & 0 & 0 & 2 \\
 4 & 2 & 1 & 0 & 0 & 3 \\
 5 & 0 & 4 & 2 & 0 & 6 \\
 6 & 4 & 5 & 0 & 0 & 9 \\
 7 & 0 & 12 & 2 & 0 & 14 \\
 8 & 8 & 12 & 0 & 0 & 20 \\
 9 & 0 & 24 & 2 & 0 & 26 \\
 10 & 16 & 22 & 2 & 0 & 40 \\
 11 & 0 & 40 & 2 & 0 & 42 \\
 12 & 32 & 40 & 0 & 0 & 72 \\
 13 & 0 & 60 & 2 & 0 & 62 \\
 14 & 64 & 57 & 2 & 0 & 123 \\
 15 & 0 & 94 & 8 & 0 & 102 \\
 16 & 128 & 80 & 0 & 0 & 208 \\
 17 & 0 & 112 & 2 & 0 & 114 \\
 18 & 256 & 147 & 2 & 0 & 405 \\
 19 & 0 & 144 & 2 & 0 & 146 \\
 20 & 512 & 144 & 12 & 0 & 668 \\
\hline
\end{tabular}
\caption{Finite-check counts for $2\le n\le 20$.  The check uses the full distance profile over nonzero shifts before assigning any class label.}
\label{tab:smalln-exhaustive-counts}
\end{table}

Rows in the comparison table use the tuple
$(\mathrm{HT},\mathrm{SP},\mathrm{SUB},\mathrm{MP},\mathrm{total})$.  The finite-check tuple is the output of the finite scan; the formula tuple is computed from the class-count theorems, the trichotomy theorem for $\mathrm{MP}=0$, and the aggregate count corollary.  Every displayed row agrees.
\begin{center}
\scriptsize
\begin{tabular}{rccc}
\hline
$n$ & finite-check tuple & formula tuple & agreement \\
\hline
 2 & $(1,0,0,0,1)$ & $(1,0,0,0,1)$ & agrees \\
 3 & $(0,2,0,0,2)$ & $(0,2,0,0,2)$ & agrees \\
 4 & $(2,1,0,0,3)$ & $(2,1,0,0,3)$ & agrees \\
 5 & $(0,4,2,0,6)$ & $(0,4,2,0,6)$ & agrees \\
 6 & $(4,5,0,0,9)$ & $(4,5,0,0,9)$ & agrees \\
 7 & $(0,12,2,0,14)$ & $(0,12,2,0,14)$ & agrees \\
 8 & $(8,12,0,0,20)$ & $(8,12,0,0,20)$ & agrees \\
 9 & $(0,24,2,0,26)$ & $(0,24,2,0,26)$ & agrees \\
10 & $(16,22,2,0,40)$ & $(16,22,2,0,40)$ & agrees \\
11 & $(0,40,2,0,42)$ & $(0,40,2,0,42)$ & agrees \\
12 & $(32,40,0,0,72)$ & $(32,40,0,0,72)$ & agrees \\
13 & $(0,60,2,0,62)$ & $(0,60,2,0,62)$ & agrees \\
14 & $(64,57,2,0,123)$ & $(64,57,2,0,123)$ & agrees \\
15 & $(0,94,8,0,102)$ & $(0,94,8,0,102)$ & agrees \\
16 & $(128,80,0,0,208)$ & $(128,80,0,0,208)$ & agrees \\
17 & $(0,112,2,0,114)$ & $(0,112,2,0,114)$ & agrees \\
18 & $(256,147,2,0,405)$ & $(256,147,2,0,405)$ & agrees \\
19 & $(0,144,2,0,146)$ & $(0,144,2,0,146)$ & agrees \\
20 & $(512,144,12,0,668)$ & $(512,144,12,0,668)$ & agrees \\
\hline
\end{tabular}
\end{center}

Across this checked range, the cumulative totals are HT $=1{,}023$, SP $=1{,}000$, SUB $=40$, MP $=0$, and total $=2{,}063$, again consistent with the aggregate count corollary and the trichotomy theorem.

\subsection{Prescribed-shift counts}
\label{sec:fixedshift}

Fix a prescribed nonzero shift $i\in\{1,\dots,n-1\}$. The main manuscript mentions prescribed-shift counts only as comparison material. Here we derive the corresponding prescribed-shift necklace count in the notation of the main manuscript and compare it with Gabri\'c's binary word-level formula \cite[Lem.~7]{Gabric2022}. This comparison stays at the prescribed-shift level and does not recover the full minimizing-shift set $\Ktwo(w)$.

Let
\[
H(n,i):=\{u\in\{0,1\}^n:d_u(i)=2\},
\qquad
h(n,i):=|H(n,i)|.
\]
The set $H(n,i)$ is the prescribed-shift binary word slice counted by Gabri\'c \cite[Lem.~7]{Gabric2022}, and $h(n,i)$ is its cardinality.  Here $d_u(i)$ is the same rotational Hamming distance used throughout the main manuscript: $d_u(i)=\Ham(u,\rot^i u)$.  Let $N(n,i)$ denote the corresponding prescribed-shift necklace count, and write $g:=\gcdop(n,i)$ and $m:=n/g$.  Writing
\[
H(n,i)\subseteq H(n):=\{u\in\{0,1\}^n:d_u(k)=2\text{ for some }k\in\{1,\dots,n-1\}\},
\]
the prescribed shift $i$ itself witnesses the condition, so $H(n,i)\subseteq H(n)$.

The derivation follows directly from the local geometry in the main manuscript.
The shift-cycle decomposition, active-orbit lemma, and block lemma in the main manuscript give the needed equivalence.
A word $u$ lies in $H(n,i)$ exactly when one of the $g$ $i$-orbits of $\Zn$ has a proper cyclic interval $I\subseteq\Zquot{m}$ as its $1$-set and the other $g-1$ orbits are constant.
Conversely, any such choice gives $d_u(i)=2$.
The active orbit contributes exactly two boundary points, and each constant orbit contributes zero.
The choices are independent: $g$ choices for the active orbit, $m(m-1)$ choices for the interval $I$ (initial point and length $b\in\{1,\dots,m-1\}$), and $2^{g-1}$ choices for the constant orbit values.
Therefore
\[
h(n,i)=g\cdot m(m-1)\cdot 2^{g-1}=2^{g-1}n(m-1).
\]
This is exactly the binary prescribed-shift specialization of \cite[Lem.~7]{Gabric2022}, stated in the same notation and used here only for comparison.
The derivation uses only the displayed lemmas cited above.
For the quotient step, we record the same prescribed-shift geometry as a local primitivity check.

\begin{lemma}[Prescribed-shift distance $2$ implies primitivity]
\label{lem:fixedshift-d2-primitive}
Let $n\ge2$, let $1\le i<n$, and let $u\in\B^n$.
If $d_u(i)=2$, then $u$ has trivial rotation stabilizer.
Consequently, every word in $H(n,i)$, and every word in $H(n)$, has a rotation orbit of size $n$.
\end{lemma}

\begin{proof}
Write $g:=\gcdop(n,i)$ and $m:=n/g$.
By the active-orbit and block lemmas in the main manuscript, the $i$-orbits contain a unique active orbit $\mathcal O$, whose $1$-set is a proper cyclic interval $I\subseteq\Zquot{m}$, while all other $i$-orbits are constant.
Suppose $\rot^s u=u$.
Translation by $s$ sends $i$-orbits to $i$-orbits and preserves their internal distance contributions, so it must send the unique active orbit to itself.
Thus $s\in\langle i\rangle$, say $s\equiv ji\pmod n$.
If $s\ne0$, then $j\not\equiv0\pmod m$, and the restriction of $u$ to $\mathcal O$ would force $I+j=I$ in $\Zquot{m}$.
But the interval-shift criterion in the main manuscript gives $|I\triangle(I+j)|>0$ for every nonzero $j\in\Zquot{m}$ and every proper cyclic interval $I$, a contradiction.
Hence $s=0$ in $\Zn$, so the stabilizer is trivial and the rotation orbit has size $n$.
The assertion for $H(n)$ follows by applying the first part to any nonzero witness shift.
\end{proof}

To pass from this word-level slice to necklace language, we use the same two ingredients as in the aggregate and odd-prime comparisons: primitivity and rotation invariance.  Every word in the prescribed-shift slice $H(n,i)$ has a distance-$2$ witness and is therefore primitive by \Cref{lem:fixedshift-d2-primitive}.  Hence each rotation orbit in the slice has size $n$.

The predicate $d_w(i)=2$ is rotation-invariant: for every $j\in\Zn$,
\[
d_{\rot^j w}(i)=\Ham(\rot^j w,\rot^{i+j}w)=\Ham(w,\rot^i w)=d_w(i).
\]
Thus $H(n,i)$ is a union of primitive rotation orbits.  Each such orbit has exactly $n$ words by \Cref{lem:fixedshift-d2-primitive}, so
\[
N(n,i)=|H(n,i)/\langle\rot\rangle|=\frac{h(n,i)}{n},
\]
and taking the lexicographically least word in each orbit would instead recover the matching Lyndon representatives.

Substituting the displayed local word count yields
\begin{equation}
\label{eq:gabric-fixedshift}
N(n,i)=2^{g-1}(m-1).
\end{equation}
Thus \eqref{eq:gabric-fixedshift} is the prescribed-shift necklace count obtained by dividing by the primitive orbit size from \Cref{lem:fixedshift-d2-primitive}.

\subsection{Aggregate and pure half-turn comparison with Gabri\'c}
\label{sec:gabric-crosswalk}

The prescribed-shift slice in \Cref{sec:fixedshift} provides one comparison.  Gabri\'c's aggregate word count \cite[Thm.~17]{Gabric2022} corresponds, after binary specialization and primitive-orbit division, to the present total necklace count.  Gabri\'c's Lyndon-word quotient \cite[Thm.~19]{Gabric2022} gives the same aggregate comparison at the Lyndon-word level.  We make these correspondences explicit as follows.

Let $h_2(n)$ denote Gabri\'c's aggregate count $h(n)$ over the binary alphabet, and set
\[
H_2(n):=\{u\in\B^n:d_u(k)=2\text{ for some }1\le k<n\}.
\]
Then
\[
H_2(n)=\{u\in\B^n:\minDist(u)=2\}.
\]
The inclusion from right to left is immediate from the definition of $\minDist$.  Conversely, if $d_u(k)=2$ for some nonzero $k$, then \Cref{lem:fixedshift-d2-primitive} shows that $u$ is primitive.  The positive-minimum primitivity lemma in the main manuscript gives $\minDist(u)>0$, and the even-distance lemma there forces $\minDist(u)=2$.

The set $H_2(n)$ is rotation-invariant.  Every orbit in it has size $n$ by \Cref{lem:fixedshift-d2-primitive}, so
\begin{equation}
\label{eq:gabric-aggregate-crosswalk}
\frac{h_2(n)}{n}=|H_2(n)/\langle\rot\rangle|=|\Dtwo(n)|.
\end{equation}
The formula-level comparison also reorganizes into the three classwise summands of the aggregate count corollary.
\paragraph{External comparison inputs.}
The rest of this subsection uses only binary specializations of Lemma~16 and Theorems~17--19 of Gabri\'c~\cite[Lem.~16, Thms.~17--19]{Gabric2022}.  These external results are used only for comparison; none is used in the main proof.

Let $p_2(r)$ denote the number of nonprimitive binary words of length $r$, so that $p_2(r)=2^r-\PrimCt{r}$ in the present notation.  For a summand in Theorem~17, let $h''_2(n,i)$ denote Gabri\'c's irredundant word-level summand over the binary alphabet, and put $g_i:=\gcdop(n,i)$.

The irredundant index convention is as follows.  Set $m:=n/g_i$ and $\tau:=n/m$.  Each $m$-block is represented by an index $i=\tau a$ with $a\in(\ZZ/m\ZZ)^\times$ and $1\le a\le m/2$; this choice selects one representative from each antipodal unit pair.

After the primitive-orbit division by $n$ used in \eqref{eq:gabric-aggregate-crosswalk}, the imported Lemma~16/Theorem~17 summand is
\begin{equation}
\label{eq:gabric-hpp-binary}
\frac{h''_2(n,i)}{n}=
\begin{cases}
\dfrac12\left(2^i\left(\dfrac ni-1\right)-2p_2(i)\right)
=\PrimCt{i}+2^{i-1}\left(\dfrac ni-3\right),& i\mid n,\\[0.7em]
2^{g_i-1}\left(\dfrac n{g_i}-3\right),& i\nmid n
\end{cases}.
\end{equation}
The imported Theorem~18 specialization for words with exactly one distance-$2$ conjugate is
\begin{equation}
\label{eq:gabric-htriple-binary}
h^{\prime\prime\prime}_2(2r)=r\bigl(2^r-2p_2(r)\bigr),
\qquad h^{\prime\prime\prime}_2(2r+1)=0,
\end{equation}
where $h^{\prime\prime\prime}_2(N)$ denotes Gabri\'c's word-level count of binary words of length $N$ with exactly one distance-$2$ conjugate.  Theorem~19 supplies the corresponding Lyndon-word quotient.  In the present distance-$2$ setting, \Cref{lem:fixedshift-d2-primitive} makes the counted words primitive, so the quotient is the same primitive-orbit division by the word length used in \eqref{eq:gabric-aggregate-crosswalk}.

Here the equality in the divisor case uses $p_2(i)=2^i-\PrimCt{i}$.  The comparison in \eqref{eq:gabric-hpp-binary} checks the HT/SP/SUB enumeration formulas listed in the introduction.  In the $m$-block, $a=1$ is the divisor case $i=\tau$.  The restriction $a\le m/2$ selects one element from each antipodal unit pair.  Therefore, \eqref{eq:gabric-hpp-binary} gives
\begin{equation}
\label{eq:gabric-block-reorg}
\sum_{\substack{1\le a\le m/2\\ \gcdop(a,m)=1}}
\frac{h''_2(n,\tau a)}{n}
=
\begin{cases}
\PrimCt{\tau}-2^{\tau-1},&m=2,\\[0.3em]
\PrimCt{\tau},&m=3,\\[0.3em]
\PrimCt{\tau}+\dfrac{\varphi(m)}2\,2^{\tau-1}(m-3),&m\ge4
\end{cases}.
\end{equation}
Summing \eqref{eq:gabric-block-reorg} over divisors $m\mid n$ with $m\ge2$ yields the manuscript's aggregate formula.  The HT contribution is
\[
\One_{2\mid n}\left(\PrimCt{n/2}-2^{n/2-1}\right)
+
\sum_{\substack{m\mid n\\ m\ge4\,\mathrm{even}}}
\PrimCt{n/m}
=
\One_{2\mid n}\left(\sum_{\tau\mid n/2}\PrimCt{\tau}-2^{n/2-1}\right)
=\One_{2\mid n}2^{n/2-1}.
\]
The $m=3$ term gives the contribution from the order-$3$ SP boundary case, $\One_{3\mid n}\PrimCt{n/3}$, and the remaining SP contribution is
\[
\sum_{\substack{m\mid n\\m\ge4}}
\frac{\varphi(m)}2\,2^{n/m-1}(m-3).
\]
Finally, the SUB contribution is
\[
\sum_{\substack{m\mid n\\m\ge5\,\mathrm{odd}}}\PrimCt{n/m}
=
\sum_{\substack{\tau\mid n\\ (n/\tau)\,\mathrm{odd}\\ n/\tau\ge5}}
\PrimCt{\tau}.
\]
Thus Gabri\'c's Theorems~17 and~19 give the aggregate word-level and Lyndon-level comparison.  In the present notation, this comparison corresponds to the aggregate count corollary.  The present paper refines it by recovering $\Ktwo(w)$ and splitting the resulting necklaces into HT, SP, and SUB classes.

Gabri\'c's exactly-one-conjugate count gives the pure half-turn boundary.  The only external input used here is the binary Theorem~18 specialization recorded in \eqref{eq:gabric-htriple-binary}; it is used only for external comparison.

Any word counted by $h^{\prime\prime\prime}_2(n)$ has a distance-$2$ conjugate.  By \Cref{lem:fixedshift-d2-primitive}, such a word is primitive, and its nontrivial shifts give distinct conjugates.  Since $d_u(k)=d_u(n-k)$ for all $k$, exactly one such conjugate can occur only when $n=2m$ and the unique active shift is the half-turn $m$.  Conversely, a word representative of a pure half-turn necklace has $\Ktwo(u)=\{m\}$ and therefore exactly one distance-$2$ conjugate.  Hence
\begin{equation}
\label{eq:gabric-pureHT-crosswalk}
\frac{h^{\prime\prime\prime}_2(2m)}{2m}=|\DtwoHT(2m;m)|.
\end{equation}
The algebra also matches the pure half-turn corollary.  Gabri\'c's notation $p_2(m)$ counts nonprimitive binary words of length $m$ (proper powers). Thus, in the present notation,
\[
p_2(m)=2^m-\PrimCt{m}.
\]
Dividing the displayed Theorem~18 specialization by the primitive orbit size $2m$ and substituting $p_2(m)=2^m-\PrimCt{m}$ yields
\[
\frac{h^{\prime\prime\prime}_2(2m)}{2m}
=\frac{2^m-2(2^m-\PrimCt{m})}{2}
=\PrimCt{m}-2^{m-1}
=|\DtwoHT(2m;m)|,
\]
which is exactly the corollary's pure half-turn formula.  This comparison remains at the aggregate and pure-boundary levels; it does not provide the HT/SP/SUB minimizer-set classifier supplied by the trichotomy theorem.